\documentclass[12pt, reqno]{amsart}

\usepackage{amssymb,amsmath,amscd,graphicx,
latexsym,amsthm}
\usepackage{amssymb,latexsym,amsmath,amscd,graphicx}
\usepackage[mathscr]{eucal}

\usepackage{scalerel,stackengine}
\DeclareMathSymbol{\shortminus}{\mathbin}{AMSa}{"39}

\stackMath
\newcommand\reallywidehat[1]{%
\savestack{\tmpbox}{\stretchto{%
  \scaleto{%
    \scalerel*[\widthof{\ensuremath{#1}}]{\kern.1pt\mathchar"0362\kern.1pt}%
    {\rule{0ex}{\textheight}}
  }{\textheight}%
}{2.4ex}}%
\stackon[-6.9pt]{#1}{\tmpbox}%
}
 \input xy
 \xyoption{all}
\def\OO{{\mathcal O}}

\def\F{\mathcal{F}}
\def\E{\mathcal{E}}
\def\G{\mathcal{G}}

\def\I{\mathcal{I}}

\def\cP{\mathcal{P}}
\def\Pic0{{\rm Pic}^0}

\def\Aut0{{\rm Aut}^0}

\def\l{{\underline l}}

\def\*{{\underline *}}
\def\utheta{{\underline \theta}}
\def\udelta{{\underline \delta}}

\def\rk{\mathrm{rk}}

\theoremstyle{plain}

\newtheorem{theorem}{Theorem}[section]
\newtheorem{theoremalpha}{Theorem}
\newtheorem{corollaryalpha}[theoremalpha]{Corollary}
\newtheorem{propositionalpha}[theoremalpha]{Proposition}
\newtheorem{proposition/example}[theorem]{Proposition/Example}
\newtheorem{proposition}[theorem]{Proposition}

\newtheorem{corollary}[theorem]{Corollary}
\newtheorem{lemma}[theorem]{Lemma}

\theoremstyle{definition}
\newtheorem*{introremark}{Remark}

\newtheorem{remark}[theorem]{Remark}

\newtheorem{conjecture/question}[theorem]{Conjecture/Question}

\newtheorem{remark/definition}[theorem]{Remark/Definition}
\newtheorem{notation/assumptions}[theorem]{Assumptions/Notation}
\newtheorem{setting/notation}[theorem]{Setting and Notation}

\numberwithin{equation}{section}

 \theoremstyle{remark}

\begin{document}

\title[Semihomogeneous vector bundles, $\mathbb Q$-twisted sheaves]{Semihomogenous vector bundles, $\mathbb Q$-twisted sheaves, duality,  and linear systems on abelian varieties}




 \author[N. Alvarado, G.Pareschi]{Nelson Alvarado and Giuseppe Pareschi}

\address{ Departamento de Matemáticas, Facultad de Ciencias, Universidad de Chile, Las Palmeras 3425,  Santiago\\Chile}
\email{nelson.alvarado@ug.uchile.cl}
\address{Department of Mathematics,
              University of Rome Tor Vergata\\Italy}
\email{pareschi@mat.uniroma2.it}
 \thanks{Both authors were partially supported by  the MIUR Excellence Department Project MatMod@TOV awarded to the Department of Mathematics of the University of Rome Tor Vergata. GP was also partially supported by the PRIN 2022 "Moduli spaces and Birational Geometry" and is a member of GNSAGA - INDAM}

\begin{abstract} In this paper we point out the natural relation between $\mathbb Q$-twisted objects of the derived category of abelian varieties,  cohomological rank functions, and semihomogeneous vector bundles.   We apply this 
 to two basic classes of objects, corresponding  to each other via the Fourier-Mukai-Poincar\'e transform:  twists of the ideal sheaf of one point and of the evaluation complexes of ample simple semihomogeneous vector bundles.  
This naturally leads  to the introduction of $\mathbb Q^{\ge 0}$- graded section modules associated to line bundles on abelian varieties, built by means of semihomogeneous vector bundles (and containing the usual section rings). We prove a duality relation between such  modules associated to dual polarizations, which is not visible at the level of the usual section rings. Other applications include formulas relating the thresholds of relevant cohomological rank functions appearing in this context.  As a consequence we show a lower bound for the base point free threshold of a polarization in function of its type, and some obstructions to surjectivity of multiplication maps of global sections of  certain line bundles.
\end{abstract}

\maketitle

After Mukai (\cite{mukai}) it has been increasingly clear that many geometric aspects of abelian  varieties can be approached by means of the Fourier-Mukai  transform associated to the Poincar\'e line bundle (FM transform for short). In this context  $\mathbb Q$-twisted sheaves (and more generally $\mathbb Q$-twisted objects of the bounded derived category)  appear naturally, giving rise to the cohomological rank functions introduced in \cite{ens} (see 
e.g.   \cite{caucci2}, \cite{ito-3-folds},  \cite{zhi-survey}, \cite{zhi3} for some geometric applications).    One of the purposes of this paper is to show the role of semihomogeneous vector bundles in this story. (Interestingly, semihomogeneous vector bundles play a prominent role also in Lahoz-Rojas' recent paper \cite{lr}, where cohomological rank functions on abelian surfaces are introduced and studied via stability conditions.)

We work with polarized abelian varieties $(A, \l)$  on an algebraically closed field $k$ of characteristic zero.  Given $\lambda  \in\mathbb Q$, we denote  $E_{A,\lambda\l}$  a simple bundle in  Mukai's category $\mathbb S_{A,\lambda\l}$ of semihomogeneous vector bundles $E$ on $A$ such that $\frac {[\det E]} {\rk E}=\lambda\l$ in  $\mathrm{NS}(A)_\mathbb Q$ (\cite[\S6]{semihom}).

   The starting point of our paper (Proposition \ref{fractional} below) is the following:  \emph{for an object $\F\in D(A)$, the object $\F\otimes E_{A,\lambda\l}$  has the  cohomological behavior \emph{(namely: has the same cohomological rank functions)} of the $\mathbb Q$-twisted object $\F^{\,\oplus r}\langle\lambda\l\rangle$, where $
 r$ is  the rank of the bundle $E_{A,\lambda\l}$}.
 In this sense one can identify  (a suitable direct sum of copies of) $\F\langle \lambda\l\rangle$ with $\F\otimes E_{A,\lambda\l}$. 
 This has the advantage of dealing with a honest and explicit sheaf (although of much higher rank, since the rank of $E_{A,\lambda\l}$ is usually huge), or more generally object of $D(A)$, rather than a somewhat elusive  $\mathbb Q$-twisted sheaf.

 Here we develop some consequences of this observation in two specific cases: (a) the ideal sheaf of a closed point  (for example the origin), twisted by $\mathbb Q$-powers of an ample polarization, (b)  the twisted evaluation complex of the global sections of  ample simple semihomogeneous vector bundles. We will show that  the FM transform induces a precise cohomological relation between them.
 
 To put the subject into perspective, we recall that the vanishing threshold for the $h^1$ of   the ideal sheaf of any closed point, twisted by rational powers of a polarization $\l$, (studied in Caucci's paper \cite{caucci}, and called there the \emph{base point freeness threshold}), measures how  far is  the polarization from having base points. By \cite[\S8]{ens} (generalized in Proposition \ref{s-beta-n} below) and \cite{caucci} such threshold is known to have a precise relation with a similar vanishing threshold for the hypercohomology of the evaluation complex of global sections of line bundles representing $\l$, which in turn measures properties as projective normality  and the $N_p$ condition of $\l$. As observed by A. Ito and Z. Jiang, the base point freeness threshold can be investigated via birational geometry methods \`a la Anghern-Siu and Helmke. This provided a novel approach to projective normality and syzygies of abelian varieties, a classical and much studied topic, leading to results which were out of reach with the previous methods (see the works of Caucci (\cite{caucci}, \cite{caucci3}, Ito (\cite{ito-0}, \cite{ito-M-reg},\cite{ito-3-folds},\cite{ito3},\cite{ito4}) and Jiang (\cite{zhi-survey}, \cite{zhi2}).
 
 The purpose of this paper is to sharpen the above ideas in two different ways: \\ 
  (a) bringing semihomogeneous vector bundles into the picture. This offers an alternative way (different from the  birational geometry methods) for investigating the base point freeness thresholds, namely  (generic) global generation criteria for semihomogeneous vector bundles, see Subsection \ref{subs:GG}.   \\
(b) Establishing a duality result between two pairs of thresholds:  on the one hand the vanishing threshold of the $h^0$ (respectively, of the $h^1$, i.e. the base point freeness threshold)  of the ideal sheaf of  a closed point of an abelian variety $A$, twisted by $\mathbb Q$-powers of a  polarization $\l$,  and, on the other hand, the vanishing threshold of the $h^1$ (resp. of the $h^0$) of the ideal sheaf of  a closed point of the dual variety $\widehat A$, twisted by $\mathbb Q$-powers of the dual polarization $\hat\l$. \\
 As application of these results:\\
 (c)  we provide some lower bounds on the base point freeness threshold. In turn these are translated, via Proposition \ref{s-beta-n} below, in obstructions to surjectivity of multiplication maps of global sections, hence, in particular, to projective normality. Interestingly,   the semihomogeneous bundles approach  brings naturally into the game the \emph{type} of the polarization $\l$. 

All this suggests the possibility of further applications of these methods, both to upper and lower bounds for the base point freeness threshold.

\subsection{$\mathbb Q$-graded modules and multiplication maps of global sections. } Let us turn to a more detailed exposition. 
 Although it will not appear explicitly in most of our results,  we will implicitly deal with  the \emph{$\mathbb Q$-graded section modules of a polarization}, defined as follows. Let $L$ be a line bundle representing the polarizaton $\l$. Let
\[ 
{\mathcal S}^\mathbb Q_{A}(L,\alpha):= \bigoplus_{\lambda\in \mathbb Q^{\ge 0}}H^0(A,E_{A,\lambda\l}\otimes P_\alpha),
\]
where $P_\alpha$ is the line bundle parametrized by $\alpha\in\widehat A:=\Pic0 A$ and $\{E_{A\lambda\l}\in\mathbb S_{A,\lambda\l}\}_{\lambda\in Q^{\ge 0}}$ is a compatible collection of simple semihomogeneous vector bundles, where compatible means that $E_{A,\lambda\l}\otimes L=E_{A,(\lambda+1)\l}$ for all $\lambda$. Clearly 
\[E_{A,n\l}=L^{\otimes n}\quad\hbox{ for $n\in \mathbb Z$}.
\]
 Therefore ${\mathcal S}^\mathbb Q_{A}(L,\OO_A)$ contains the section ring 
\[
\mathcal S_A(L)=\bigoplus_{n\in\mathbb N}H^0(A,L^{\otimes n})
\] and, as well as all $ {\mathcal S}^\mathbb Q_{A}(L,\alpha)$'s, it is a graded module over it. The interest of the graded modules
$ {\mathcal S}^\mathbb Q_{A}(L,\alpha)$ is that, in the above sense, they can be seen as representing such things as "modules of sections  of fractional powers of $L$".

We study the structure maps $\mathcal S_A(L)\otimes_k {\mathcal S}^\mathbb Q_{A}(L,\alpha)\rightarrow {\mathcal S}^\mathbb Q_{A}(L,\alpha)$, and specifically their graded components, namely the multiplication maps of global sections
\[
H^0(A,L^{\otimes n})\otimes H^0(A,E_{A,x\l}\otimes P_\alpha)\rightarrow H^0(A,E_{A, (n+x)\l}\otimes P_\alpha),
\]
which, as we will see, can be seen as "fractional" analogues of the usual multiplication maps 
\[H^0(A,L^{\otimes n})\otimes H^0(A,L^{\otimes k}\otimes P_\alpha)\rightarrow H^0(A,L^{\otimes n+k}\otimes P_\alpha).\footnote{This was also the point of view of the second author's paper \cite{n-torsion}, where a bound to the number of $n$-torsion points on theta divisors of p.p.a.v.s. was obtained by an analysis of the maps $m^{1/(n-1)}_{A,n\utheta,\alpha}$ (in the notation (\ref{eq:multifrac}) below).}
\]
%
Our analysis shows that, even if our  interest is in the properties of the section  ring $\mathcal S(A,L)$ (projective normality, syzygies..), the FM transform leads us to consider the  modules $ {\mathcal S}^\mathbb Q_{A}(L,\alpha)$. For example, denoting $M$ a line bundle  on the dual variety $\widehat A$ representing the dual polarization $\hat\l$, there is a precise  duality relation between the modules  $ {\mathcal S}^\mathbb Q_{A}(L,\alpha)$ and  $ {\mathcal S}^\mathbb Q_{\widehat A}(M,x)$, which is   not visible at the level of the section rings $\mathcal S(A,L)$ and $S(\widehat A,M)$. 

More generally, we will study the multiplication maps of global sections of simple semihomogeneous vector bundles
\begin{equation}\label{eq:multifrac}
m_{A,\nu\l,\alpha}^{y}: H^0(A,E_{A,\nu\l})\otimes H^0(A,E_{A, y\nu\l}\otimes P_\alpha)\rightarrow H^0(A,E_{A,\nu\l}\otimes E_{A,y\nu\l}\otimes P_\alpha).\end{equation}

\subsection{Thresholds. } We will be mainly concerned with the following injectivity and surjectivity thresholds for the multiplication maps $m^{y}_{A,\nu\l, \alpha}$ (see Subsection \ref{subs:thre} below):
\begin{align*}
 s^0_A(\nu\l)&=\sup\{y\in\mathbb Q\>|\>m^{y}_{A,\nu\l, \alpha}\> \hbox{is injective} \>\forall \alpha\in\widehat A\},\\
 s^1_A(\nu\l)&=\inf\{y\in\mathbb Q\>|\>m^{y}_{A,\nu\l,\alpha}\> \hbox{is surjective} \>\forall \alpha\in\widehat A\}.
\end{align*}

\begin{introremark}The conditions $s^1_A(\l)<1$ (respectively $s^1_A(\l)>1$) mean  that the multiplication maps 
\begin{equation}\label{solite}
m_{A,\l,\alpha}^{1}:H^0(A,L)\otimes H^0(A,L\otimes P_\alpha)\rightarrow H^0(L^{\otimes 2}\otimes P_\alpha)
\end{equation}
are surjective (respectively, are not surjective) for all $\alpha\in \widehat A$, and they imply that the polarization $\l$ is (resp. is not) projectively normal. On the other hand $s^1_A(\l)=1$ means that the maps (\ref{solite}) are not surjective for $\alpha$ belonging to a proper closed subset $Z$ of the dual variety $\widehat A$. If this is the case, projective normality means that the identity point $\hat e$ does not belong to $Z$, and this may or may not happen. (See Subsection \ref{subs:GG} below). 
 For example: if $\utheta$ is a principal polarization then $s^1_A(2\utheta)=1$ but clearly $2\utheta$ is not projectively normal. By a result of Rubei (\cite{rubei}), the same happens for polarizations $2\l$, with $\l$ of type $(1,\dots , 2)$. On the other hand, as it will follow from Corollary \ref{cor} below, polarizations $\l$ of type $(1,4,\dots ,4)$ have $s^1_A(\l)\ge 1$. Based on computer calculations, \cite{fuentes} asserts that in dimension $3$ and $4$ the general such polarizations are projectively normal (therefore in particular $s^1_A(\l)=1$).  

\end{introremark}

The  other thresholds we are interested in are
\begin{align*}
 \beta^0_A(\l)&=\sup\{\lambda\in\mathbb Q\>|\> \I_e\langle\lambda\l\rangle\>\hbox{is IT(1)}\}\\
\beta^1_A(\l)&=\inf\{\lambda\in\mathbb Q\>|\> \I_e\langle\lambda\l\rangle\>\hbox{is IT(0)}\}.
\end{align*}
where $e$ is the identity point in $A$. For the notation about $\mathbb Q$-twisted sheaves and the vanishing conditions IT(i) (meaning "satisfies the index theorem with index $i$") we refer to Subsections \ref{subs: not(d)} and \ref{subs:Q}. Briefly, they are the extension to $\mathbb Q$-twisted sheaves of the well known index conditions for coherent sheaves on abelian varieties.  
The invariant $\beta^1_A(\l)$  is the above mentioned \emph{base point freeness threshold} (the relation with base point freeness stems from the fact that in any case $\beta^1_A(\l)\le 1$ and equality holds if and only if the line bundles representing $\l$ are not base point free).

\subsection{Summary of results. } 

\S2 provides some precise relations between semihomogeneous vector bundles, $\mathbb Q$-twisted sheaves, and cohomological rank functions. 

In \S3 we give an intepretation  of the thresholds $\beta^i_A(\l)$ in terms of semihomogeneous bundles. For example, in Proposition \ref{interpret} we show that, for $\lambda\in \mathbb Q^+$, we have that $\beta^1_A(\l)\le \lambda$ (resp. $\beta^1_A(\l)<\lambda$) if and only if the vector  bundle $E_{A,\lambda\l}$ is generically globally generated (resp. globally generated). 

Next, we prove the following duality result, obtained via the Fourier-Mukai transform, relating the above  thresholds of a given polarized abelian variety $(A,\l)$ and of its dual $(\widehat A,\hat\l)$ (\cite[\S 14,3]{birke-lange}, see also Subsection \ref{subs: not(b)} below).  We denote 
\[(d_1,\dots ,d_g)
\]  
the type of $\l$ (hence the type of $\hat\l$ is $(d_1,d_1d_g/d_{g-1},\dots ,d_1d_g/d_2,d_g)$).

\begin{theoremalpha}\label{duality1} For $i=0,1$
\begin{equation}\label{duality1-1}
\beta^i_A(\l)=\frac 1 {\beta^{1-i}_{\widehat A}(\hat\l)d_1d_g}\>.
\end{equation}
\end{theoremalpha}
\noindent(This a corollary of the more general Lemma \ref{duality} and Theorem \ref{exchange2} below, where the duality is formulated at the level of Fourier-Mukai transform and cohomological rank functions).

It is well known from \cite{ens} that, via the Fourier-Mukai transform,  for integral values of $\nu$ the  thresholds $s^i_A(\nu\l)$  and the threshold $\beta^i_A(\l)$ are related by a precise formula. In \S\ref{sect:multip} we generalize this to rational values of $\nu$, proving the following Proposition (Corollary \ref{cor:fract}).
\begin{propositionalpha}\label{s-beta-n} For $\nu\in\mathbb Q$, $\nu>\beta^i_A(\l)$
\[s^i_A(\nu\l)=\frac {\beta^i_A(\l)} {\nu-\beta^i_A(\l)}.
\]
\end{propositionalpha} 
 (We point out that also this is a consequence of a more general  result relating cohomological rank functions, see Proposition \ref{prop:main-rel-pre}).  



 As a consequence of these results we obtain the  already mentioned    duality formula
 relating the dimensions of kernel and cokernel of suitable graded components of the structure maps of the $\mathbb Q^{\ge 0}$-graded section modules associated to dual polarizations (Corollary \ref{expression2}). 

Going back to Theorem \ref{duality1}, combining it with Proposition \ref{s-beta-n} we have the following relation between the threshold $\beta^1$ of a polarization and the threshold $s^0$ of rational multiples of the dual polarization
\begin{corollaryalpha}\label{corollary} For every $\nu\in\mathbb Q^+$,
\begin{equation}\label{duality1-2}
\beta^1_A(\l)=\frac{s^0_{\widehat A}(\nu\hat\l)+1} {s^0_{\widehat A}(\nu\hat\l)d_1d_g\nu}\>.
\end{equation}
\end{corollaryalpha}


 According to Corollary \ref{corollary} a surprising approach for producing a lower bound for the threshold $\beta_A^1(\l)$ -- i.e. an obstruction to base point freeness and, via Proposition \ref{s-beta-n}, to \emph{surjectivity} of multiplication maps of global sections -- is to look for  upper bounds for the threshold $s^0_{\widehat A}(\nu\hat\l)$, i.e.  obstructions to \emph{injectivity} of multiplication maps of global sections on the dual variety. From basic facts of generic vanishing theory it follows that this can be done simply by finding   semihomogeneous vector bundles (not necessarily simple) on $\widehat A$, say $F\in\mathbb S_{\widehat A,\nu\hat\l}$ and  $G\in \mathbb S_{\widehat A, y\nu\hat\l}$,  such that a multiplication map of global sections
$m^y_{\widehat A,\nu\hat l,  a}$   (see (\ref{eq:multifrac})) is not injective for some $a\in A$. 


An elementary way to do that is as follows: as soon as $E_{\widehat A,\nu\hat\l}$ is generically globally generated, letting $r=\rk 
E_{\widehat A,\nu\hat\l}$,   the image of the linear space $\wedge^{r+1}H^0(\widehat A, E_{\widehat A,\nu\hat\l})$ in $ H^0(\widehat A, E_{\widehat A,\nu\hat\l})\otimes H^0(\widehat A, \wedge^{r}E_{\widehat A,\nu\hat\l})$ is a non-trivial subspace mapping to zero via the multiplication map of global sections. Clearly the outcome of this approach depends on the rank $r$. Therefore we are lead to consider  the \emph{rank function} of a given polarization $\l$:
\[
r_{A, \l}:\mathbb Q_{>0}\rightarrow \mathbb N, \quad  r_{A,\l}(\nu)=\rk\, E_{A,\nu\l}.
\] 
In the this way we prove the following lower bound for the base point freeness threshold
  (we refer also to Corollary \ref{cor:precise} below  for a stronger variant).

\begin{theoremalpha}\label{obstruction} In the above setting and notation
\begin{equation}\label{sup}
\beta^1_A(\l)\ge \sup_{\nu>\frac 1 {\sqrt[g]{\chi(\hat\l)}}}\>\Bigl\{\frac 1 {d_1d_g\nu}\bigl(1+\frac 1 {r_{\widehat A,\hat\l}(\nu)}\bigr)\Bigr\}.
\end{equation}
\end{theoremalpha}

In turn, the key point in applying this result is that, by a  result of Mukai, \cite[Theorem 7.11(5)]{semihom}, the rank  function $r_{\widehat A,\hat\l}$ is  determined by the type of $\l$. In fact, writing in lowest terms $\nu=\frac a b$, as soon as $d_g$ is coprime with $b$ the rank of $E_{A,\nu\l}$ is very big, namely $b^g$, but it is much smaller if $gcd(b,d_i)>1$ for low $i$ (see Subsection \ref{subs:semihom}).  
In this way Theorem \ref{obstruction} provides a bound which in certain cases is meaningful, and sometimes sharp.  A more systematic analysis  seems possible and this would provide other non-trivial upper bounds of the base-point freeness threshold.

Particular cases of Theorem \ref{obstruction} are found in the following
\begin{corollaryalpha}\label{cor} Let $(d_1,\dots ,d_g)$ the type of $\l$. 
\begin{enumerate}
\item  If $d_g>1$ then 
\[
\beta^1_A(\l)\ge \frac 2 {d_g}. \quad \hbox{i.e. \ } s^1_A(\l)\ge \frac 2 {d_g-2}.
\]
Particular cases
 \begin{enumerate}
 \item (\cite[Corollary 2.6]{nr}) If $d_g= 2$  then 
 \[\beta_A^1(\l)=1 \quad\hbox{i.e. \ $\l$ has base points};
 \]
 \item if $d_g=3$
 then 
 \[
 \beta^1_A(\l)\ge \frac 2 3\quad\hbox{i.e.\ } s^1_A(\l)\ge 2;
 \]

\item If $d_g=4$  then  
\[\beta^1_A(\l)\ge \frac 1 2\quad\hbox{i.e. \ } s^1_A(\l)\ge 1.
\]
\end{enumerate}

 \item If $(d_1,\dots ,d_g)=(1,n,\dots n,nm) $ with $m<n$ then
 \[
\beta^1_A(\l)\ge \frac {1+m}{mn}\quad\hbox{i.e. \ } s^1_A(\l)\ge\frac{1+m}{(n-1)m-1}.
\]
 In particular if  $(d_1,\dots d_g)=(1,3,\dots,3,6)$ then 
\[\beta^1_A(\l)\ge \frac 1 2\quad\hbox{i.e. \ } s^1_A(\l)\ge 1.
\]

\end{enumerate}
\end{corollaryalpha}

 The equivalence of the formulations in terms of $\beta^1_A(\l)$ and in terms of $s^1_A(\l)$ follows from Proposition \ref{s-beta-n} for $\nu=1$ (which was already in \cite[Corollary 8.2]{ens}). We recall that, as pointed out in the above Remark, the condition $s^1_A(\l)\ge 1$ does not say that $\l$ is not projectively normal, but only that there are $\alpha\in\widehat A$ such that the multiplication map $m^1_{A,\l\alpha}$ of (\ref{solite}) is not surjective. Similarly, the condition $s^1_A(\l)\ge 2$ means only that for some $\alpha\in\widehat A$ the map $m^2_{A,\l,\alpha}$ is not surjective. A better understanding of this point is an open problem (see Subsection \ref{base-loci}). 
For example the previous result concerning polarizations of type $(1,3,\dots,3,6)$  does not imply that such polarizations are not projective normal.  This is still unknown even in low dimension, where computer calculations are available  (\cite{fuentes}).

Finally, the last section is about possible natural developments of the present work.


\section{Notation and background material. }\label{notation} 

We work on an algebraically closed ground field $k$ of characteristic zero. 

\subsection{Translation and multiplication by an integer. }\label{subs: not(a)} Given an abelian variety $A$, and a closed point $x\in A$, we denote $t_x:A\rightarrow A$ the translation by $x$. 
Moreover, given a positive integer $b$ we denote
\[
b_A:A\rightarrow A
\]
the multiplication-by-$b$ isogeny and $A[b]$ its kernel, i.e. the group of $b$-division points of $A$.

\subsection{Dual variety and Fourier-Mukai-Poincar\'e transform. }\label{subs: not(b)} We denote $\widehat A:=\Pic0 A$ the dual abelian variety and $\mathcal \cP$ the Poincar\'e line bundle on $A\times\widehat A$. Given $\alpha\in\widehat A$ (respectively $x\in A$), the corresponding line bundle on $A$ (resp. on $\widehat A$) is denoted $P_\alpha$ (resp. $P_x$). The Fourier-Mukai equivalences associated to  $\cP$ (see \cite{mukai})  are denoted
\[
\Phi_\cP:D(A)\rightarrow D(\widehat A)\qquad \Psi_{\cP}:D(\widehat A)\rightarrow D(A)
\]
(where $D( \cdot )$ is the bounded derived category). Similarly, we denote $\Phi_{\cP^\vee}$ and $\Psi_{\cP^\vee}$ the Fourier-Mukai equivalences associated to $\cP^\vee$.  A basic property (see (\cite[(3.1)]{mukai}) is that 
\begin{equation}\label{translation}
\Phi_{\cP}(t_x^*\F)=\Phi_{\cP}(\F)\otimes P_x.
\end{equation}
 The following formula (\cite[Lemma 2.1]{pp2}) will be useful. For all objects $\F\in D(A)$ and $\G\in D(\widehat A)$ 
\begin{equation}\label{exchange}
H^i(A,\F\>\underline\otimes\>\Psi_{\cP}(\G))\cong H^i(\widehat A, \Phi_{\cP}(\F)\>\underline\otimes\>\G).
\end{equation}

Finally, using the Fourier-Mukai transform we will sometimes use the dualization functor in the following form:
\begin{equation}\label{dualization}
\F^\vee:=R\mathcal Hom(\F,\OO_A).
\end{equation}

\subsection{Polarizations. }\label{subs: not(c)} Given an ample line bundle $L$ on $A$ we will denote $\l$ its class in the N\'eron-Severi group, and 
\[\varphi_{\l}:A\rightarrow \widehat A
\]
the corresponding isogeny (namely $\varphi_\l(x)=t_x^*L\otimes L^{-1}$). The type of $\hat\l$ is  denoted $(d_1,\dots ,d_g)$. \\
The \emph{dual polarization of $\l$} (see \cite[\S 14.3]{birke-lange}) is  denoted $\hat\l$. It is a polarization on $\widehat A$  of type $(d_1,d_1d_g/d_{g-1},\dots ,d_1d_g/d_2,d_g)$. Therefore
\begin{equation}\label{dual1}
\chi(\l)\chi(\hat \l)=(d_1d_g)^g.
\end{equation}
Other useful properties (found in \cite[Proposition 14.4.1]{birke-lange}) are
\begin{equation}\label{dual2}\varphi_{\hat\l}\varphi_\l=(d_1d_g)_A
\end{equation}
\begin{equation}\label{dual3}
\varphi_\l^*\hat\l=d_1d_g\l \>.
\end{equation}

\subsection{Vanishing conditions. } \label{subs: not(d)} 
Given a coherent sheaf $\F$ on $A$, or, more generally, an object of $D(A)$, its cohomological support loci are the subvarieties
\begin{equation}\label{eq:coho}
V^i(A,\F)=\{\alpha\in\widehat A\>|\> h^i(A,\F\otimes P_\alpha)>0\}.
\end{equation}
 Given $i\in \mathbb Z$, $\F$ is said to \emph{satisfy the index theorem with index $i$ (or to be IT(i) for short) } if 
 \[V^j(A,\F)=\emptyset\quad\hbox{for all $j\ne i$. }
 \]
  If this happens then  $\Phi_{\cP}(\F)=R^i\Phi_{\cP}(\F)[-i]$ and the sheaf $R^i\Phi_{\cP}(\F) $ is locally free. 
 If this is the case we will adopt the following notation
\begin{equation}\label{hat}
\widehat\F:=R^i\Phi_{\cP}(\F).
\end{equation}
Another type of vanishing condition is \emph{generic vanishing}. A coherent sheaf $\F$ (or, more generally, an object of $D(A)$ such that $V^i(A,\F)=\emptyset $ for $i<0$)  is said to be $GV$ if the following condition holds
\[\mathrm{codim}_{\Pic0 A}V^i(A,\F)\ge i\quad\hbox{for all $i\ge 0$.}
\]

\subsection{Simple semihomogeneous vector bundles. }\label{subs:semihom} Semihomogeneous  vector bundles on an abelian variety $A$ are  vector bundles $E$ such that for each $x\in A$ there exists a point $\alpha \in \Pic0 A$ such that $t_x^*E\cong E\otimes P_\alpha$. 
The main reference for this topic  is the fundamental paper of Mukai \cite{semihom}. Here we recall the results of Mukai needed in the sequel. As above, let $L$ be a line bundle on an abelian variety $A$, $\l$  its N\'eron-Severi class, and $\lambda$ a rational number. Following Mukai, we consider  semihomogeneous vector bundles $E$ such that 
\[\frac{[\det E]}{rk(E)}=\lambda\l.
\] The set of such vector bundles is denoted $\mathbb S_{A,\lambda \l}$ (\cite[\S 6]{semihom}). 
We are mainly interested in \emph{simple} vector bundles in $\mathbb S_{A,\lambda \l}$.  

\noindent
\emph{(a) } For all $\lambda\in \mathbb Q$ and $\l$ in $NS(A)$ there exists a simple vector bundle $E_{A,\lambda\l}$ in $\mathbb S_{A,\lambda\l}$ (\cite[Theorem 7.11(1)]{semihom}). Such simple vector bundles are unique up to tensorization with a $P_\alpha$, for $\alpha\in\widehat A$ (\cite[Proposition 6.17]{semihom}).

\noindent \emph{(b) } Write $\lambda=\frac a b$ (with $b>0$) and denote $A[b]$ the group of $b$-divison point in $A$. Let 
\[
u_{A,\l}^2(a,b):=ord(A[b]\cap K(a\l)).
\]
  Then, by \cite[Theorem 7.11(5)]{semihom}, we have that
\begin{equation}\label{rk/dim}
r_{A,\l}(\lambda):= \rk \, E_{A,\lambda\l}=\frac {b^g} {u_{A,\l}(a,b)}, \qquad \chi(E_{A,\lambda\l})=\frac {\chi(a\l)} {u_{A,\l}(a,b)}=\frac{a^g\chi(\l)}{u_{A,\l}(a,b)}.
 \end{equation} 
 
 \noindent \emph{(c) }  There exists an isogeny $\pi:B\rightarrow A$ and a line bundle $M$ on $B$ such that $\pi^*E_{A,\lambda\l}\cong M^{\oplus r}$ (\cite[Proposition 7.3]{semihom}). \\
 By  \cite[Proposition 7.6(2)]{semihom} such isogenies are exactly those factoring trough $p:\Phi(E_{A,\lambda\l})\rightarrow A$, where, given  any vector bundle $E$ on an abelian variety $A$, $\Phi(E)\subset A\times \widehat A$ is   the subgroup associated  to  $E$ in \cite[Definition 3.5]{semihom}, and $p$ is the first projection restricted to $\Phi(E)$. In turn, \cite[Theorem 7.11(3)]{semihom} asserts that, writing $\lambda=\frac a b$,
\begin{equation}\label{Phi}
\Phi(E_{A,\lambda\l})={\rm Im}(b_A,\varphi_{a\l}):A\rightarrow A\times A^\vee.
\end{equation}

\noindent \emph {(d)} Any $F\in\mathbb S_{A,\lambda_\l}$, not necessarily simple, is of the form $F\cong \oplus_i F_i$ where $F_i\in \mathbb S_{A,\lambda\l}$ is a \emph{$E_{A,\lambda\l}\otimes P_{\alpha_i}$-potent bundle} (where $\alpha_i\in\Pic0 A$), i.e. has a filtration $0\subset G_1\subset \cdots\subset G_{k-1}\subset G_k=F_i$ with $G_j/G_{j-1}\cong E_{A,\lambda\l}\otimes P_{\alpha_i}$,  (\cite[Proposition 6.18, Theorem 7.11(2)]{semihom}). 

\subsection{$\mathbb Q$-twisted sheaves (or complexes) and cohomological rank functions. }\label{subs:Q} $\mathbb Q$-twisted coherent sheaves (\cite[Definitions 6.2.1 and 6.2.2]{laz2}) are defined as follows:  one consider pairs $(\F,\udelta)$, where $\F$ is a coherent sheaf on an abelian variety $A$ and $\udelta$ is a class in $\mathrm{NS}(A)_{\mathbb Q}$. Two such pairs are identified under the equivalence relation generated by $(\F\otimes L,\delta)\sim (\F,\l+\udelta)$ (where $L$ is a line bundle and $\l$ its class in $\mathrm{NS}(A)$). The equivalence class of $(\F,\udelta)$ is called a \emph{$\mathbb Q$-twisted sheaf} and denoted $\F\langle\udelta\rangle$. Similarly, given an object $\F\in D(A)$,  the $\mathbb Q$-twisted element $\F\langle\udelta\rangle$ is defined in the same way.

 We consider a class $\udelta=\lambda\l$, where $\lambda=\frac a b\in\mathbb Q$ (with $b>0$) and $\l$ is an integral class.  For an element $\F\in D(A)$ the $i$-th (generic) \emph{(hyper)cohomology rank} of $\F\langle x\l\rangle$ is defined as follows (\cite{ens}):
\begin{equation}\label{gen}
h^i_{gen}(A,\F\langle \lambda\l\rangle):=\frac1 {b^{2g}} h^i(A, (b_A^*\F)\otimes L^{\otimes ab}\otimes P_\alpha)
\end{equation}
 (here $h^i(A, \,\cdot\,)$ means the dimension of the (hyper)cohomology vector space $H^i(A,\,\cdot\,)$), where $L$ is a line bundle representing $\l$ and for $\alpha\in \Pic0A$ is general  so that  $h^i(A, (b_A^*\F)\otimes L^{\otimes ab}\otimes P_\alpha)$ is minimal. It is easily seen that this is well defined, in the sense that it does not depend on the representation $\lambda=\frac a b$ nor on the representing line bundle $L$ (\cite{ens}). We write   
\[
h^i_{A,\F,\l}(\lambda):= h^i_{gen}(A,\F\langle \lambda\l\rangle).
\]
This defines continuous functions $h^i_{A,\F,\l}:\mathbb Q\rightarrow \mathbb Q^{\ge 0}$, (extending to continuous real functions, \cite[Theorem 3.2]{ens}). By definition, for a positive integer $n$ we have that $h^i_{A,\F,n\l}(\lambda) =h^i_{A,\F,\l}(n\lambda)$. 


\subsection{Vanishing conditions for $\mathbb Q$-twisted sheaves (or objects)}\label{IT} The rational numbers (\ref{gen}) are well defined, i.e. independent on the representation $\lambda=\frac a b$, only for $\alpha$ general in $\widehat A$, i.e. when $h^i(A, (b_A^*\F)\otimes L^{\otimes ab}\otimes P_\alpha)$ is the minimal value for $\alpha\in\widehat A$. Concerning the \emph{i-th  jump loci},  i.e. the loci  
of $\alpha\in \widehat A$ such that $h^i(A, (b_A^*\F)\otimes L^{\otimes ab}\otimes P_\alpha)$ is not minimal (in particular, the cohomological jump loci of Subsection \ref{subs: not(d)}), they  are not well defined, but their dimension is. 

In particular, it makes sense  to to say that: a \emph{ $\mathbb Q$-twisted  element of $D(A)$, say $\F\langle \frac a b\l\rangle$, is IT(i)}  if $(b_A^*\F)\otimes L^{\otimes ab}$ is so.  Similarly,  $\F\langle \frac a b\l\rangle$ is said to be GV if $(b_A^*\F)\otimes L^{\otimes ab}$ is so.
Concerning this, it is known by \cite[Theorem 5.2]{ens} (extending results of \cite{hacon}, see also \cite{pp2})  that: \\
(a)  to be IT(0) is an open-below condition, i.e. $\F\langle\lambda\l \rangle$ is IT(0) if and only if $\F\langle(\lambda-\varepsilon)\l \rangle$ is $IT(0)$ for sufficiently small $\varepsilon$,\\
(b)  a $\mathbb Q$-twisted sheaf  $\F\langle\lambda\l \rangle$ (or, more generally, a $\mathbb Q$-twisted object whose $i$-th cohomological support loci are empty for $i<0$) is $GV$ if and only if  $\F\langle(\lambda+\varepsilon)\l \rangle$ is IT(0) for all $\varepsilon>0$.



\section{$\mathbb Q$-twisted objects, semihomogeneous vector bundles and cohomological rank functions}\label{sect:2}

\subsection{Relation between $\mathbb Q$-twisted objects and semihomogeneous vector bundles. } 
The relation between simple semihomogeneous vector bundles and $\mathbb Q$-twisted sheaves is in the following slight precisation of (c) of Subsection \ref{subs:semihom}, implicit in \cite{semihom} but not explicitly stated. 

\begin{proposition}\label{basic}  Let $\lambda=\frac a b\in\mathbb Q$, let $E_{A,\lambda\l}$ be a simple bundle in $\mathbb S_{A,\lambda \l}$ and let $r_{A,\l}(\lambda)=\rk E_{A,\lambda\l}$. Then 
\begin{equation}\label{basic1}
b_A^*E_{A,\lambda\l}\cong (L^{\otimes ab})^{\oplus r_{A,\l}(\lambda)},
\end{equation}
where $L$ is a line bundle representing $\l$.
\end{proposition}
\begin{proof} It follows from (c) of Subsection \ref{subs:semihom} that
$b_A^*E_{A,\lambda\l}\cong M^{\oplus r_{A,\l}(\lambda)}$ for some line bundle $M$ on $A$. 
It follows also that $M$ has to be proportional to $\det(E)$ and hence that the class of $M$ has to be a multiple of $\l$. Hence 
\[
b_A^*E_{A,\lambda\l}\cong (L^{\otimes k})^{\oplus r_{A,\l}(\lambda)},
\]
for some integer $k$,  where $L$ is a line bundle representing the class $\l$. We have that $[\det E_{A,\lambda\l}]= \frac{r_{A,\l}(\lambda)a} b \l$ and  $[b_A^*\det E_{A,\lambda\l}]=b^2[\det E_{A,\lambda\l}]$. From this it follows that $k=ab$. 
\end{proof}

As announced, it turns out that the $\mathbb Q$-twisted sheaf $\OO_A^{\oplus r_{A,\l}(\lambda)}\langle\lambda\l\rangle$ behaves cohomologically as the bundle   $E_{A,\lambda\l}$, in the following sense

\begin{proposition}\label{fractional} We keep the notation  of Proposition \ref{basic} and of Subsection \ref{subs:Q}. For all object $\F$ in $D(A)$ and for all $\lambda,t\in\mathbb Q$ and for all simple bundle $E_{A,\lambda\l}\in\mathbb S_{A,\lambda l}$
\[ h^i_{A,\F\otimes E_{A,\lambda\l}\,,\l}(t)=r_{A,\l}(\lambda)\, h^i_{A,\F,\l}(\lambda+t).
\]
Moreover a $\mathbb Q$-twisted object $\F\langle\lambda\l\rangle$ is IT(i),  or GV, \emph{(see Subsection \ref{IT})} if and only if the object $\F\otimes\E_{A,\lambda\l}$ is so. 
\end{proposition}
\begin{proof} Let $\lambda=\frac a b$ and $t=\frac c d$. Combining Proposition  \ref{basic}  and Lemma  \ref{isogenies} below we have that 
\[
h^i_{A,\F\otimes E_{A,\lambda\l}\,,\l}(t)=\frac  {r_{A,l}(\lambda)} {b^{2g}}\, h^i_{A, b_A^*\F\otimes L^{\otimes ab}, b^2\l}(t).
\]
 The right hand side is equal to
\[\frac  {r_{A,l}(\lambda)} {(db)^{2g}}\, h^i_{gen}(A, (db)_A^*\F\otimes L^{\otimes (d^2ab+b^2cd)})=r_{A,\l}(\lambda)\, h^i_{A,\F,\l}(\lambda+t).
\]
\end{proof}
In the following Lemma we record the fact that, essentially by definition, the cohomology rank functions are multiplicative with respect to isogenies:
\begin{lemma} \label{isogenies} Let $\psi:B\rightarrow A$ an isogeny. Then, in the above setting and notation,
\[
h^i_{B,\psi^*\F,\psi^*\l}=(\deg\psi)\, h^i_{A,\F,\l}.
\]
\end{lemma}
\begin{proof} As usual, let $\lambda=\frac{a}{b}$. Since $b_B\psi=\psi b_A$ we have that
\begin{equation}\label{isogeny1}
b^{2g}h^i_{B,\psi^*\F,\psi^*\l}(\lambda)=h^i_{gen}(B, \psi^*(b_A^*\F\otimes L^{\otimes ab}))=h^i_{gen}(A,(\psi_*\OO_B)\otimes b_A^*\F\otimes L^{\otimes ab})
\end{equation}
where  the last inequality follows from  projection formula. Denoting $\hat\psi:\widehat A\rightarrow \widehat B$ the dual isogeny, the last member of (\ref{isogeny1}) is equal to
\[
\sum_{\alpha^\prime\in\ker\hat\psi}h^i(A,b_A^*\F\otimes L^{\otimes ab}\otimes P_{\alpha+\alpha^\prime})=(\deg \psi)\, h^i_{gen}(A,b_A^*\F\otimes L^{\otimes ab})=b^{2g}(\deg\psi)\, h^i_{A,\F,\l}(\lambda)
\]
where in the first member $\alpha\in\Pic0 A$  is sufficiently general.
\end{proof}

\begin{remark}\label{rem:nuovo} From the proof of Proposition \ref{fractional} it follows that its thesis is satisfied by every semihomogeneous vector bundle, not necessarily simple, satisfying (\ref{basic1}). Therefore: let $\lambda=\frac a b$ and $E\in \mathbb S_{A,\lambda\l}$ such that $b_A^*E\cong (L^{ab})^{\oplus r}$ (where $r=\rk E$). Then $h^i_{A,\F\otimes E\,,\l}(t)=r \, h^i_{A,\F,\l}(\lambda+t)$
\end{remark}


\subsection{Semihomogeneous vector bundles and the Fourier-Mukai transform. } In this subsection we describe the behavior of ample semihomogeneous vector bundles  (and their duals) with respect to the Fourier-Mukai transform. The following proposition is standard (compare e.g. with \cite[\S1.4]{oprea})

\begin{proposition}\label{FMP-semihom} Let $\l$ be a polarization on an abelian variety $A$ of dimension $g$, of type $(d_1,\dots ,d_g)$, and let $E\in \mathbb S_{A, \lambda\l}$. If $\lambda\in\mathbb Q^+$ (respectively $\lambda\in\mathbb Q^-$) then $E$ is IT(0) (resp. IT(g)) and $\widehat E\in \mathbb S_{\widehat A,-\frac 1 {d_1d_g\lambda}\hat\l}$ \emph{(see (\ref{hat}) for the notation)}. 
\end{proposition}
\begin{proof} We know that $E$ is a direct sum of $E_{A\lambda\l}\otimes P_{\alpha_i}$-potent vector bundles,  where $E_{A\lambda\l}$ is as usual a reference simple vector bundle in $\mathbb S_{A, \lambda\l}$ and $\alpha_i\in\Pic0 A$, (Subsection \ref{subs:semihom}(d)). Therefore
we can assume that $E$ is simple. Then the assertion about the vanishing conditions is obvious, for example by Proposition \ref{basic} and projection formula. Concerning the Fourier-Mukai transform, in the first place it is clear that $\widehat E$ is semihomogeneous (and simple). Indeed,
denoting, as in \cite[Definition 3.1]{semihom} $\Phi^0(E)=\{(x,\alpha)\in A\times \widehat A\>|\>t_x^*E\cong E\otimes P_\alpha\}$, it is clear from (\ref{translation}) that the algebraic groups $\Phi^0(E)$ and  $\Phi^0(\widehat E)$ are isomorphic. Hence the claim follows from \cite[Proposition 5.1]{semihom} asserting that a vector bundle on an abelian variety $A$, say $F$, is semihomogeneous if and only if $\dim \Phi^0(F)=\dim A$. 
The rest of the statement is proved  as follows. In the first place, by duality one can assume that $\lambda\in\mathbb Q^+$. If this is the case the Fourier-Mukai  transform exchanges $\chi$ with the rank, i.e. the rank of $\widehat E$ is $h^0(E)=\chi(E)$, while $h^g(\widehat E)=(-1)^g\chi(\widehat F)$, is the rank of $E$. By Subsection \ref{subs:semihom}(b), 
\[\frac {\chi(E)} {\rk(E)}=(\lambda)^g\chi(\l).
\]
 Therefore 
\[\frac {\chi(\widehat E)} {\rk(\widehat E)}=-\frac 1 {\lambda^g\chi(\l)}=(-\frac 1 {\lambda d_1d_g})^g\chi(\hat\l).
\]
where the last equality follows from (\ref{dual1}). Therefore, again from  Subsection \ref{subs:semihom}(b), $\widehat E\in\mathbb S_{\widehat A,-\frac 1 {d_1d_g\lambda}\hat\l}$. 
\end{proof}

As a consequence we have a different proof of the fundamental formula for the cohomological rank functions of the FMP transform \cite[Proposition 2.3]{ens}.
\begin{proposition} \label{FMP} 
Let $A$ an abelian variety and let $\F \in D(A$). For $\lambda\in\mathbb Q^{-}$ 
\begin{equation}\label{FMP1}h^i_{A,\F,\l}(\lambda)=\chi(\l)(-\lambda)^g h^i_{\widehat A, \Phi_\cP(\F),\hat\l}\,(-\frac{1}{d_1d_g\lambda}).
\end{equation}
For $\lambda\in\mathbb Q^{+}$ 
\begin{equation}\label{FMP2}h^i_{A,\F,\l}(\lambda)=\chi(\l)\lambda^g h^{g-i}_{\widehat A, \Phi_{\cP^\vee}(\F^\vee),\hat\l}\,(\frac{1}{d_1d_g\lambda}).
\end{equation}
\end{proposition}
\begin{proof} Let $\lambda\in\mathbb Q^{-}$ and let $E:=E_{\widehat A,-\frac 1 {d_1d_g\lambda}\hat\l }\in\mathbb S_{\widehat A, -\frac 1 {d_1d_g\lambda}\hat\l }$ be a simple vector bundle (on $\widehat A$).  By the previous proposition $E$ is IT(0), i.e. $\Psi_{\cP}(E)=\widehat{E}[0]$. By (\ref{translation}) and (\ref{exchange}) we have
\begin{equation}\label{PP1}
H^i(A, \F\otimes \widehat{E}\otimes P_\alpha )=H^i(\widehat A, \Phi_{\cP}(\F)\otimes t_\alpha^*E).
\end{equation}
By Proposition \ref{FMP-semihom} $\widehat E\otimes P_\alpha\in \mathbb S_{A,\lambda\l}$. Hence, applying Proposition \ref{fractional} (for $t=0$) to both sides of the previous equality, we have that
\[
r_{\widehat A,\hat\l}(-\frac 1 {d_1d_g\lambda})\, h^i_{A,\F,\l}(\lambda)=r_{A,\l}(\lambda) h^i_{\widehat A,\Phi_{\cP}(\F),\hat\l}(-\frac 1 {d_1d_g\lambda}).
\]
Then (\ref{FMP1}) follows from the formula for the ranks in Subsection \ref{subs:semihom}(b), or, more simply, because $r_{\widehat A,\hat\l}(-\frac 1 {d_1d_g\lambda})=(-1)^g\chi(E_{A,\lambda\l})$ and, again, applying Subsection \ref{subs:semihom}(b). 

The proof for   $\lambda\in\mathbb Q^+$ is similar. Let  $E\in\mathbb S_{\widehat A, \frac 1 {d_1d_g\lambda}\hat\l }$ be a simple vector bundle.  In the same way, by (\ref{translation}) and (\ref{exchange})  we have
\begin{equation}\label{PP2}H^i(A, \F\otimes \widehat E\otimes P_{-\alpha} )=H^i(\widehat A, \Phi_{\cP}(\F)\otimes t_{-\alpha}^*E).
\end{equation}
By duality (see (\ref{dualization}) for the notation) we have that 
\begin{equation}\label{eq:PP3}
H^i(A, \F\otimes \widehat E\otimes P_{-\alpha} )\cong H^{g-i}(A, \F^\vee\otimes (\widehat E)^\vee\otimes P_\alpha)^\vee.
\end{equation}
The second equality  follows from Proposition \ref{fractional} as above, because  $(\widehat E)^\vee\otimes P_\alpha$ is a simple bundle in $ \mathbb S_{A, \lambda\l}$. 
\end{proof}

Note that the equalities (\ref{FMP1}) and (\ref{FMP2}) are equivalent to those of \cite[Proposition 2.3]{ens}. Namely,  for $\lambda\in\mathbb Q^-$:
\begin{equation}\label{transf1} 
h^i_{A,\F,\l}(\lambda)=\frac {(-\lambda)^g}{\chi(\l)}h^i_{A, \varphi_\l^*\Phi_{\cP}(\F),\l}\,(-\frac 1 \lambda )
\end{equation}
and for $\lambda\in\mathbb Q^{+}$
\begin{equation}\label{transf2} 
h^i_{A,\F,\l}(\lambda)=\frac {\lambda^g}{\chi(\l)}h^{g-i}_{A, \varphi_\l^*\Phi_{\cP^\vee}(\F^\vee),\l}\,(\frac 1 \lambda ).
\end{equation}
Indeed (\ref{FMP1}) is equivalent to (\ref{transf1}) because
\[
h^{j}_{A, \varphi_\l^*\Phi_{\cP^\vee}(\F^\vee),\l}\,(-\frac 1 \lambda )=h^{j}_{A, \varphi_\l^*\Phi_{\cP^\vee}(\F^\vee),\varphi_\l^*\hat\l}\,(-\frac 1 {d_1d_g\lambda} )=\chi(\l)^2h^{j}_{\widehat A, \Phi_{\cP^\vee}(\F^\vee),\hat\l}\,(-\frac 1 {d_1d_g\lambda} ),
\]
where the first equality follows from (\ref{dual3}) and the second one follows from Proposition \ref{isogenies}, recalling that the degree of the isogeny $\varphi_\l$ is $\chi(\l)^2$. The equivalence of the equalities (\ref{FMP2}) and (\ref{transf2}) is proved in the same way.

\begin{remark}\label{jump-locus} Besides the fact that this proof  of the above equalities  is slightly simpler than the original one in \cite{ens}, an advantage of the present approach to cohomology of $\mathbb Q$-twisted sheaves  is that equalities (\ref{PP1}), (\ref{PP2}) and (\ref{eq:PP3}) hold for all points $\alpha\in\widehat A$, not only for the general points. Although we will not do this in this paper, we point out that this makes possible a more refined analysis of how the Fourier-Mukai functor transforms the cohomological jump loci of $\mathbb Q$-twisted sheaves. Interesting particular cases of cohomological jump loci are "generalized base loci". We refer to Subsection \ref{base-loci} below for more on this.
\end{remark}

\begin{remark}\label{rem:fract-polarization}[Cohomological rank functions of fractional polarizations] Obviously one can  formally define cohomological rank functions for abelian varieties equipped with a \emph{fractional polarization} $\nu\l$ with $\nu\in\mathbb Q^+$ as follows
\[h^i_{A,\F,\nu\l}(\lambda):=h^i_{A,\F,\l}(\nu\lambda).
\]
Note that, except for (\ref{transf1}) and (\ref{transf2}), where the isogeny associated to an (integral) polarization plays a role, all the  results and formulas of the previous sections involving cohomological rank functions (namely  Lemma \ref{isogenies}, Proposition \ref{FMP}
) 
can be formally extended to this setting. In order to do that, one can formally define the type of $\nu\l$ as $(\nu d_1,\dots ,\nu d_g)$ and  the dual polarization $\widehat{\nu\l}:=\nu\hat\l$. Then, for example, the transformation formula (\ref{FMP1}) extends as
\[
h^i_{A,\F,\nu\l}(\lambda)=\chi(\nu\l)(-\lambda)^g h^i_{\widehat A, \Phi_\cP(\F),\widehat{\nu\l}}\,(-\frac{1}{\nu^2d_1d_g\lambda}).
\]
Indeed 
\begin{align*}h^i_{A,\F,\nu\l}(\lambda)=h^i_{A,\F,\l}(\nu\lambda)
&\buildrel{(\ref{FMP1})}\over=\chi(\l)(-\nu\lambda)^g h^i_{\widehat A, \Phi_\cP(\F),\hat{\l}}\,(-\frac{1}{d_1d_g\nu\lambda})\\
&=\chi(\l)(-\nu\lambda)^g h^i_{\widehat A, \Phi_\cP(\F),\nu\hat{\l}}\,(-\frac{1}{d_1d_g\nu^2\lambda})\\
&=\chi(\nu\l)(-\lambda)^g h^i_{\widehat A, \Phi_\cP(\F),\widehat{\nu\l}}\,(-\frac{1}{(\nu d_1)(\nu d_g)\lambda}).
\end{align*}
\end{remark}


\section{$\mathbb Q$-twists of the ideal  sheaf of one point and duality}

\subsection{The cohomological rank functions and their thresholds. }\label{subs:interp} We will consider the cohomological rank functions of the ideal sheaf of a closed point of a given polarized abelian variety $(A,\l)$. For simplicity we take the origin $e$  as closed point (however the functions are independent on this choice). 

For $\lambda\in\mathbb Q^+$ the cohomological jump loci for the $i$-th cohomology of the $\mathbb Q$-twisted sheaf $\I_e\langle\lambda\l\rangle$ are emply for $i\ne 0,1$ (see the terminology in Subsection \ref{IT}). Moreover, it is easy to see that if $\I_e\langle\lambda\l\rangle$ is IT(0) (respectively IT(1)) then the same holds for all $\lambda^\prime\ge \lambda$ (resp. $\lambda^\prime\le\lambda$). Therefore it is natural to consider the two thresholds
\[
\beta^0_{A}(\l)=\mathrm{sup}\{\lambda\in\mathbb Q^+\>|\>\I_e\langle\lambda\l\rangle \>\hbox{is IT(1)}\}\qquad \beta^1_{A}(\l)=\mathrm{inf}\{\lambda\in\mathbb Q^+\>|\>\I_e\langle\lambda\l\rangle \>\hbox{is IT(0)}\}
\]

\begin{remark}\label{GV} 
It is easy to see that $\beta_{A}^{1}(\l)=\inf\{\lambda\in\mathbb Q^+\>|\> h^1_{A,\I_e,\l}(\lambda)=0\}.$  Therefore $\beta_{A}^{1}(\l)$ coincides with the threshold $\beta(\l)$ defined at the beginning of \cite[\S8]{ens}. 
We recall also that, given $\bar\lambda\in\mathbb Q^+$,  
\begin{equation}\label{bar-lambda}\bar\lambda=\beta^1_A(\l)
\end{equation} 
 if and only if the cohomological support locus of $\I_e\langle\bar \lambda\l\rangle$ for the $1$-th cohomology   is a non-empty proper subvariety of $\widehat A$.  Because of the vanishing of the $h^i$ for $i\ge 2$, this means that $\I_e\langle\bar\lambda\l\rangle$ is GV (but not IT(0)). In turn this is equivalent to (\ref{bar-lambda})  by (2) of Subsection \ref{IT}. 
 
 Also the similar thing for $\beta^0_A(\l)$ holds, namely that  $\beta_{A,}^{0}(\l)=\sup\{\lambda\in\mathbb Q^+\>|\> h^0_{A,\I_e,\l}(\lambda)=0\}$. We leave this to the interested reader.
\end{remark}
\subsection{Interpretation in terms of semihomogeneous bundles. }\label{subs:interpretation} In view of Proposition \ref{fractional}, semihomogeneous vector bundles allow  explicit interpretations of cohomological properties of $\mathbb Q$-twisted sheaves. In the case of the ideal sheaf of a point this is as follows. As usual, given a polarized abelian variety $(A,\l)$, and a positive rational number $\lambda$, let $E_{A,\lambda\l}$ be any simple vector bundle in $\mathbb S_{A,\lambda\l}$.  Let us consider the evaluation map 
\[
\mathrm{ev}_{A,\lambda\l}: H^0(A,E_{A,\lambda\l})\otimes\OO_A\rightarrow E_{A,\lambda\l}.
\]
 Then: \emph{ for $i=0$ (respectively, $i=1$) the value $h^i_{A,\I_e,\l}(\lambda)$ is the generic rank of the kernel (resp. cokernel)  of the  map $\mathrm{ev}_{A,\lambda\l}$, normalized by the rank $r_{A,\l}(\lambda)$. } Indeed, by Proposition \ref{fractional}, $h^i_{A,\I_e,\l}(\lambda)=\frac 1 {r_{A,\l}(\lambda)}h^i_{A,\I_e\otimes E_{A,\lambda\l},\l}(0)$.  From  the exact sequences
\begin{equation}\label{ex-seq}
0\rightarrow \I_e\otimes E_{A,\lambda\l}\otimes P_\alpha\rightarrow E_{A,\lambda\l}\otimes P_\alpha\rightarrow E_{A,\lambda\l}\otimes P_\alpha\otimes k(e)\rightarrow 0, 
\end{equation}
 recalling that the vector bundles $E_{A,\lambda\l}$ are IT(0) for $\lambda>0$, it follows that, for $i=0,1$, the cohomological ranks $h^i_{A,\I_e\otimes E_{A,\lambda\l},\l}(0)$ are respectively the dimension of the kernel and of the cokernel of the evaluation map  $H^0(A, E_{A,\lambda\l}\otimes P_\alpha)\rightarrow H^0(A, E_{A,\lambda\l}\otimes P_\alpha\otimes k(e))$ for general $\alpha\in\widehat A$. By Remark \ref{rem:translations} below, this map coincides with the evaluation map $H^0(A, t_x^*E_{A,\lambda\l})\rightarrow H^0(A, t_x^* E_{A,\lambda\l}\otimes k(e))$, i.e. the evaluation map  $H^0(A, E_{A,\lambda\l})\rightarrow H^0(A, E_{A,\lambda\l}\otimes k(-x))$, for general $x\in A$. 
 
 \begin{remark}\label{rem:translations}
It is useful to keep in mind the well known fact  that, for $\lambda\ne 0$, all simple vector bundles $E_{A,\lambda\l}\otimes P_\alpha\in\mathbb S_{A,\lambda\l}$ are translates $E_{A,\lambda\l}$. This follows from (c)  of Subsection \ref{subs:semihom} (or Proposition \ref{basic}). 
\end{remark}

In particular, the thresholds $\beta^i_A(\l)$ can be interpreted as follows 

\begin{proposition}\label{interpret} Let $\lambda\in\mathbb Q^+$. Then
\begin{enumerate} 
\item $\lambda>\beta^1_A(\l)$ if and only if $E_{A,\lambda\l}$ is globally generated. 
\item  $\lambda=\beta^1_A(\l)$ if and only if $E_{A,\lambda\l}$ is generically globally generated but not globally generated.
\item $\lambda<\beta^0_A(\l)$ if and only if every global section of $E_{A,\lambda\l}$ is nowhere vanishing.
\item $\lambda=\beta^0_A(\l)$ if and only  there are zeroes of global sections of $E_{A,\lambda\l}$ but they don't fill $A$, i.e. the evaluation map 
\begin{equation}\label{eq:ev}H^0(A,E_{A,\lambda\l})\otimes\OO_A\rightarrow E_{A,\lambda\l}
\end{equation}
 is injective as a map of sheaves (but not as a map of vector bundles). 
 \end{enumerate}
 \end{proposition}
 \begin{proof} (2) follows directly from the above interpretation and Remark \ref{GV}. (1) follows  in the same way. Items (c) and (d) are similar.
 \end{proof}

\begin{remark}\label{rem:equality} Clearly $\beta^0_A(\l)\le \beta^1_A(\l)$, and it can happen that equality holds. For example this is the case for principal polarizations, where in fact $\beta^0_A(\l)= \beta^1_A(\l)=1$. In general,  if  there is a $\lambda\in\mathbb Q^+$ such that the evaluation map $\mathrm{ev}_{A,\lambda\l}$ is injective and generically surjective  then $\beta^0_A(\l)= \beta^1_A(\l)=\lambda$. In turn this yields that $\chi(E_{A,\lambda\l})=rk E_{A,\lambda\l}$, i.e., using Subsection \ref{subs:semihom}(b), $\lambda=\frac 1{\sqrt[g]{\chi(\l)}}$ (where $g=\dim A$).     (To this purpose we recall that in  \cite[Theorem A(1)]{rojas} Rojas shows that, for polarized abelian surfaces $(A,\l)$ with $\l$ of type $(1,d)$ such that $\l\cdot D$ is a multiple of $\l^2$ for every divisor class $D$, if $\sqrt{\chi(\l)}$ is integral then $\beta^0_A(\l)= \beta^1_A(\l)=\frac 1 {\sqrt{\chi(\l)}}$.)
\end{remark}

\subsection{Duality. } A key point for the results of this paper is the following simple lemma. For an abelian variety $A$ we denote $e$ its origin and $\hat e$ the origin of $\widehat A$. The corresponding skyscraper sheaves are denoted $k(e)$ and $k(\hat e)$. We use the notation (\ref{dualization}) for the derived dual.
\begin{lemma}\label{duality} $\Phi_{\cP}(\I_e^\vee)=I_{\hat e}[-g+1].$
\end{lemma}
\begin{proof} Dualizing the trivial exact sequence $0\rightarrow\I_e\rightarrow\OO_A\rightarrow k(e)\rightarrow 0$ we get that 
\[
\mathcal{E}xt^i(\I_e,\OO_A)=\begin{cases}\OO_A&\hbox{for} \> i=0\\ k(e)&\hbox{for}\>i=g-1\\ 0&\hbox{otherwise.}\end{cases}
\]
 It is well known that $\Phi_{\cP}(\OO_A)=k(\hat e)[-g]$ (\cite[Proof of Theorem in \S 13]{mumford}), and it is clear that $\phi_{\cP}(k(e))=\OO_{\widehat A}[0].$ Therefore the spectral sequence $R^p\Phi_{\cP}(\mathcal{E}xt^q(\I_e,\OO_A)\Rightarrow R^{p+q}\Phi_{\cP}(\I_e^\vee)$ collapses giving the exact sequence $0\rightarrow R^{g-1}\Phi_{\cP}(\I_e^\vee) \rightarrow \OO_{\widehat A}\rightarrow k(\hat e)\rightarrow 0$ and the vanishings $R^{i}\Phi_{\cP}(\I_e^\vee)=0$ for $i\ne g-1$.
\end{proof}
Plugging Lemma \ref{duality} into the second formula of Proposition \ref{FMP} we get the following fundamental symmetry between the cohomological rank functions of the ideal of one point on an abelian variety and on its dual. 

\begin{theorem}\label{exchange2} For $\lambda\in\mathbb Q^+$ and $i=0,1$
\[h^i_{A,\I_e,\l}(\lambda)=\chi(\l)\lambda^gh^{1-i}_{\widehat A,\I_{\hat e},\hat\l}(\frac 1 {d_1d_g\lambda}).
\]
Therefore 
\[\beta^i_A(\l)=\frac 1 {d_1d_g\beta^{1-i}_{\widehat A}(\hat\l)}.
\]
\end{theorem}
\begin{remark}\label{rem:fract2} Also the previous Theorem formally extends to the setting of fractional polarizations (see Remark \ref{rem:fract-polarization}), giving
\[h^i_{A,\I_e,\nu\l}(\lambda)=\chi(\nu\l)\lambda^gh^{1-i}_{\widehat A,\I_{\hat e},\widehat{\nu\l}}(\frac 1 {\nu^2d_1d_g\lambda}).
\]
\end{remark}

\section{Multiplication maps of global sections and their thresholds. }\label{sect:multip}


 
\subsection{Relation between cohomological rank functions. } Here we will establish  a relation between the cohomological rank functions of the ideal of a point and the ones of the evaluation complex of a suitable simple semihomogeneous vector bundle. This  generalizes  \cite[\S8]{ens} (with a slightly simpler treatment).

\vskip0.2truecm We have  the following Fourier-Mukai lemma.

\begin{lemma}\label{lem:fractional} Let $E_{A,\nu\l}$ be a simple semihomogeneous vector bundle in $\mathbb S_{A,\nu \l}$, where $\nu=\frac a b$, with $a,b>0$. Applying the functor $\varphi_{a\l}^*\Phi_{\cP}$ to the restriction map
\[
E_{A,\nu\l}\rightarrow E_{A,\nu\l}\otimes k(e)
\]
one gets the twisted evaluation map
\[b_A^*H^0(A,E_{A,\nu\l})\otimes L^{\otimes -ab}\rightarrow \OO_A^{\oplus r_{A,\l}(\nu)}\cong (b_A^*E_{A,\nu\l})\otimes L^{\otimes -ab}
\]
where $r_{A,\l}(\nu)=\rk E_{A,\nu\l}$ and $L$ is a line bundle representing $\l$ such that $b_A^*E\cong (L^{\otimes ab})^{\oplus r_{A,\l}(\nu)}$ \emph{(Proposition \ref{basic})}. 
\end{lemma}
\begin{proof} We know that $E_{A,\nu\l}$ is an IT(0) vector bundle.  We claim that 
\begin{equation}\label{eq:pullback}
\varphi_{a\l}^*\widehat{E_{A,\nu\l}}\cong ({L^\prime}^{\otimes -ab})^{\oplus h^0(A,E_{A,\nu\l})} 
\end{equation}
 where  $L^\prime$ is a line bundle representing $\l$.  To prove this we argue as in the proof of Proposition \ref{basic}. We have that  $\rk \widehat{E_{A,\nu\l}}=h^0(A,E_{A,\lambda\l})$. Moreover
 $\widehat{E_{A,\nu\l}}$ is simple and belongs to $\mathbb S_{\widehat A,-\frac 1{d_1d_g\nu} \hat \l}$ (Proposition \ref{FMP-semihom}).  Therefore, according to Subsection \ref{subs:semihom}(c), we have that  the subgroup $\Phi(\widehat{E_{A,\nu\l}})\subset \widehat A\times A$ is the image of the morphism $((ad_1d_g)_{\widehat A},\varphi_{b\hat\l}):\widehat A\rightarrow \widehat A\times A$, and  any isogeny $\psi:B\rightarrow \widehat A$  factoring trough the restriction to $\Phi(\widehat{E_{A,\nu\l}})$  of the first projection of $\widehat A\times A$ is such that $\psi^*\widehat{E_{A,\nu\l}}\cong M^{\oplus h^0(A,E_{A,\nu\l})}$ for a line bundle $M$. Since  $(ad_1d_g)_{\widehat A}=\varphi_{a\l}\varphi_{\hat \l}$ ((\ref{dual2})), it is easily seen that in fact $\varphi_{a\l}:A\rightarrow \widehat A$ satisfies such a condition. Therefore  $\varphi_{a\l}^*\widehat{E_{A,\lambda\l}}\cong M^{\oplus h^0(A,E_{A,\nu\l})} $. To see who is $M$, recall that
 by Proposition \ref{basic}, we know that $(ad_1d_g)_{\widehat A}^*\widehat{E_{A,\nu\l}}\cong ({L^\prime}^{\otimes -abd_1d_g})^{\oplus h^0(A,E_{A,\nu\l})}$, where $L^\prime$ is a line bundle representing $\hat\l$.  From this, using again that $(ad_1d_g)_{\widehat A}=\varphi_{a\l}\varphi_{\hat \l}$ and that $\varphi_{\hat\l}^*\l=d_1d_g\hat\l$ ((\ref{dual3})),  the claimed (\ref{eq:pullback}) follows. 

Next, we need to understand the induced map
 \begin{equation}\label{eq:morphism}
 \varphi_{a\l}^*(\widehat{E_{A,\nu\l}})\rightarrow \varphi_{a\l}^*(\widehat {E_{A,\nu\l}\otimes k(0)}).
 \end{equation} 
Concerning the target, we have the isomorphism 
 $\varphi_{a\l}^*(\widehat {E_{A,\nu\l}\otimes k(0)}) = \OO_A^{\oplus r_{A,\l}(\nu)}$.  By (\ref{eq:pullback}), the map (\ref{eq:morphism}) is identified to a morphism $ ({L^\prime}^{\otimes -ab})^{\oplus h^0(A,E_{A,\nu\l})}\rightarrow    \OO_A^{\oplus r_{A,\l}(\nu)}$. Recalling how the Fourier-Mukai transform operates, we have that (\ref{eq:morphism}) is an evaluation map of global sections of $E_{A,\nu\l}$ tensored by ${L^\prime}^{\otimes -ab}$. Since we know by Proposition \ref{basic} that $b_A^*E_{A,\nu\l}\otimes L^{\otimes -ab}$ is trivial (for a suitable line bundle $L$ representing $\l$), it follows that (\ref{eq:morphism}) is identified to a twisted evaluation map 
 $ b_A^*H^0(A,E_{A,\nu\l})\otimes {L^\prime}^{\otimes -ab}\rightarrow b_A^*E_{A,\nu\l}\otimes L^{\otimes -ab}$. It follows that $L^{\otimes -ab}={L^\prime}^{\otimes -ab}$ and the proof is complete.
 \end{proof}
Let us denote $EV_{A, \nu\l}^\bullet$ the complex 
\begin{equation}\label{eq:ev}
0\rightarrow H^0(A,E_{A,\nu\l})\otimes \OO_A \rightarrow E_{A,\nu\l}\rightarrow 0
\end{equation}
(see Remark \ref{abuse} below about the notation).
The basic result we are interested in  is the following relation between the cohomological rank functions of the ideal sheaf of the origin and of the above complex. We use the notation of Remark \ref{rem:fract-polarization} about cohomological rank functions associated to fractional polarizations.

\begin{proposition} \label{prop:main-rel-pre} Let $\nu\in\mathbb Q^+$ and $\lambda\in (0,\nu)\cap \mathbb Q$. Then, in the notation of Remark \ref{rem:fract-polarization},
\begin{equation}\label{eq:h-nu-pre}
h^i_{A,\I_e,\nu\l}(\lambda)=\frac{(1-\lambda )^g} {\chi(\nu\l)r_{A,\l}(\nu)}\ h^i_{A, EV_{\nu\l}^\bullet,\nu\l}(\frac{\lambda}{1-\lambda}).
\end{equation}
\end{proposition}
\begin{proof}  Let $t\in\mathbb Q\cap(-1,0)$ and write $\nu=\frac a b$.  We have that 
\begin{equation}\label{eq:1}
h^i_{A,\I_e,\nu\l}(1+t)=h^i_{A,\I_0,\l}(\nu(1+t))=\frac 1 {r_{A,\l}(\nu)} h^i_{A,\I_0\otimes E_{A,\nu\l},\l}(\nu  t )=\frac 1 {r_{A,\l}(\nu)} h^i_{A,\I_0\otimes E_{A,\nu\l},a\l}(\frac t b),
\end{equation}
where the second equality follows from Proposition \ref{fractional}.  Since $\I_e\otimes E_{A,\nu\l}$ is isomorphic (in $D(A)$) to the complex $0\rightarrow E_{A,\nu\l}\buildrel{res}\over\rightarrow E_{A,\nu\l}\otimes k(e)\rightarrow 0$, it follows from Lemma \ref{lem:fractional} that $\varphi_\l^*\Phi_{\cP}(\I_e\otimes E_{A,\nu\l})$ is isomorphic, in $D(A)$, to the complex $(b_A^*EV_{A,\nu\l}^\bullet)\otimes L^{-ab}$.
 Thus   the  Fourier-Mukai transform formula (\ref{transf1})  yields that
\begin{equation}
h^i_{A,\I_0\otimes E_{A,\nu\l},a\l}(\frac t b)=\frac{(-\frac t b )^g} {\chi(a\l)} h^i_{A, (b_A^*EV_{A,\nu\l}^\bullet)\otimes L^{\otimes -ab},a\l}(-\frac b t ).
\end{equation}
Manipulating the right hand side we get
\begin{align*}
\frac{(-\frac t b )^g} {\chi(a\l)} h^i_{A, (b_A^*EV_{A,\nu\l}^\bullet)\otimes L^{\otimes -ab},a\l}(-\frac b t )&=\frac{(-\frac t b )^g} {\chi(a\l)} h^i_{A, b_A^*EV_{A, \nu\l}^\bullet,a\l}(-\frac b t-b )\\
&=\frac{(-\frac t b )^g} {\chi(a\l)} h^i_{A, b_A^*EV_{A,\nu\l}^\bullet,\l}(a(-\frac b t-b ))\\
&=\frac{(-\frac t b )^g} {\chi(a\l)} b^{2g} h^i_{A, EV_{A,\nu\l}^\bullet,\l}(\frac a{b^2}(-\frac b t-b ))\\
&=\frac{(- t  )^g} {\chi(\nu\l)} h^i_{A, EV_{A, \nu\l}^\bullet,\l}(-\frac{\nu(1+t)} t).\\
\end{align*}
Combining with (\ref{eq:1}) we get
\[
h^i_{A,\I_e,\nu\l}(1+t)= \frac{(- t  )^g} {\chi(\nu\l)r_{A,\l}(\nu) } h^i_{A, EV_{A,\nu\l}^\bullet,\nu\l}(-\frac{1+t} t).
\]

The Proposition follows letting $\lambda=1+t$. 
\end{proof}

\subsection{Interpretation in terms of semihomogeneous vector bundles. }\label{interpret-multi} By Proposition \ref{fractional}, (\ref{eq:h-nu-pre}) can be equivalently written as follows
\begin{equation}\label{eq:h-nu}
h^i_{A,\I_e,\nu\l}(\lambda)=\frac{(1-\lambda )^g} {\chi(\nu\l)r_{A,\l}(\frac{\lambda}{1-\lambda})r_{A,\l}(\nu)}\ h^i_{A, EV_{A, \nu\l}^\bullet\otimes E_{A,\frac{\lambda}{1-\lambda} \nu\l},\nu\l}(0)
\end{equation}

For $\lambda\in \mathbb Q\cap (0,\nu)$ both the vector bundles  $E_{A,\frac \lambda {1-\lambda}\nu\l}$ and  $E_{A,\nu\l}\otimes E_{A,\frac \lambda {1-\lambda}\nu\l}$ are IT(0). Therefore the appropriate spectral sequence computing the hypercohomology of the complex $EV_{A, \nu\l}^\bullet\otimes E_{A,\frac \lambda {1-\lambda}\l}\otimes P_\alpha$ collapses to the cohomology of the complex
\begin{equation}\label{complex}
\xymatrix{
0\rightarrow H^0(A, E_{A,\nu\l})\otimes H^0(A, E_{A,\frac \lambda {1-\lambda}\nu\l}\otimes P_\alpha)\ar[rr]^{m^{\frac \lambda {1-\lambda}}_{A, \nu\l, \alpha}}&& H^0(A, E_{A,\nu\l}\otimes E_{A,\frac \lambda {1-\lambda}\nu\l}\otimes P_\alpha)\rightarrow 0.}
\end{equation}
where the notation for the multiplication map of global sections  is the one of the Introduction, see (\ref{eq:multifrac}) .

Hence the cohomology ranks $h^i_{A, EV_{A,\nu\l}^\bullet\otimes E_{A,\frac{\lambda}{1-\lambda} \nu\l},\nu\l}(0)$, $i=0,1$ are nothing else than the dimensions of the kernel and of the cokernel of the multiplication map of global sections
$m^{\frac \lambda {1-\lambda}}_{A, \nu\l, \alpha}$ for generic $\alpha\in\widehat A$. 

\begin{remark}\label{abuse} The notation $m^{y}_{A,\nu\l,\alpha}$ for the maps of multiplication of global sections, as well as $EV_{A,\nu\l}^\bullet$ for the evaluation complex (\ref{eq:ev}), are in fact a simplifying  abuse, because both the maps and the complexes do depend on the representing semihomogeneous bundles chosen. However the cohomological rank functions built on them, as well as the thresholds, do not depend on that.
\end{remark}

Combining (\ref{eq:h-nu}) for $\nu=1$ and Corollary \ref{exchange2} we get the following Corollary, which ultimately can be seen as a duality relation between the $\mathbb Q$-graded modules $\mathcal S_A^\mathbb Q(L,\alpha)$ and  $\mathcal S_{\widehat A}^\mathbb Q(M,x)$  mentioned in the Introduction (where $L$ and $M$ are line bundles respectively on $A$ and $\widehat A$ representing the polarizations $\l$ and $\hat\l$), not holding at the level of section rings $\mathcal S_A(L)$ and $\mathcal S_{\widehat A}(M)$. 

\begin{corollary}\label{expression2} 
\[
\frac{(1-\lambda )^g} {\chi(\l)r_{A,\l}(\frac{\lambda}{1-\lambda})}\ h^i_{A, EV_{A, \l}^\bullet\otimes E_{A,\frac{\lambda}{1-\lambda} \l},\l}(0)
=
\frac{\chi(\l)\lambda^g(\frac{d_1d_g\lambda-1}{d_1d_g\lambda})^g}{\chi(\hat\l)r_{\widehat A,\hat\l}(\frac 1 {d_1d_g\lambda-1})} h^{1-i}_{\widehat A,EV^\bullet_{\widehat A,\hat\l}\otimes E_{\widehat A,\frac 1 {d_1d_g\lambda-1}\hat\l},\hat\l} (0).
\]

\end{corollary}

\subsection{Thresholds and relation between them. }\label{subs:thre} At this point it is natural to consider the injectivity and surjectivity thresholds for multiplication maps of global sections mentioned in the Introduction, namely
\[
s^0_A(\nu\l)=\sup\{y\in\mathbb Q\>|\>m^{y}_{A,\nu\l,\alpha}\> \hbox{is injective} \>\> \forall \alpha\in\widehat A\}, 
\]
\[
s^1_A(\nu\l)=\inf\{y\in\mathbb Q\>|\>m^{y}_{A,\nu\l,\alpha}\> \hbox{is surjective} \>\> \forall \alpha\in\widehat A\}. 
\]
\begin{remark}\label{rem:forgotten-remark}  In view of Subsection \ref{interpret-multi} the condition that the map $m^{y}_{A,\nu\l,\alpha}$ is injective (respectively surjective) for all $\alpha\in \widehat A$ is equivalent to the fact that the complex $EV^\bullet_{A,\nu\l}\otimes E_{A,y\nu\l}$ is IT(1) (resp. IT(0)). In turn, by the last part of Proposition  \ref{fractional}, this means that the $\mathbb Q$-twisted object
$EV^\bullet_{A,\nu\l}\langle y\nu\l\rangle$ is IT(1) (resp. IT(0)). Therefore all results of Subsection \ref{IT} apply. Hence, as in Remark \ref{GV}, the above thresholds can be equivantly defined as follows
\[
s^0_A(\nu\l)=\sup\{y\in\mathbb Q\>|\>h^0_{A,EV^\bullet_{A,\nu\l}.\nu\l}(y)=0\}, \qquad 
s^1_A(\nu\l)=\inf\{y\in\mathbb Q\>|\>h^1_{A,EV^\bullet_{A,\nu\l}.\nu\l}(y)=0\}. 
\]
\end{remark}

As a consequence of  Proposition \ref{prop:main-rel-pre} we have Proposition \ref{s-beta-n} of the Introduction:

\begin{corollary}\label{cor:fract} For $\nu> \beta^i_A(\l)$
\begin{equation}\label{s-beta-nu-rip}
s^i_A(\nu\l)=\frac {\beta^i_A(\l)} {\nu-\beta^i_A(\l)}.
\end{equation} 
\end{corollary}

\begin{remark}\label{rem:infty}
The above Corollary is valid also for $\nu=\beta^i_A(\l)$ letting $s^i_A(\nu\l)=+\infty$. In fact, by (c) and (d) of Proposition \ref{interpret}, if $\nu\le\beta^0_A(\l)$ then $s^0_A(\nu\l)=+\infty$. Indeed the evaluation map of global sections of the bundle $E_{A,\nu\l}$ is injective (as map of sheaves), hence all multiplication maps of global sections $m_{A,\nu\l, \alpha}^{y}$ are injective for all $y>0$ and for all $\alpha\in \widehat A$.   
Similarly, by (a) and (b) of Proposition \ref{interpret}, if $\nu\le\beta^1_A(\l)$ then $s^1_A(\nu\l)=+\infty$. Indeed $E_{A,\nu\l}$ is not globally generated and, as a consequence of Serre's theorems it follows that for $y$ sufficiently big no multiplication map of global sections $m_{A,\nu\l,\alpha}^{y}$ can be surjective. 
\end{remark}

\begin{remark}\label{rem:any} This remark will be useful in the next section. It follows from Remark \ref{rem:nuovo} that  (\ref{eq:h-nu}), (\ref{complex}) and  in the definition of the thresholds $s_{A}^i(\nu\l)$, the simple vector bundle $E_{A,y\nu\l}$
appearing as second factor of the source of the multiplication maps $m^{y}_{A,\nu\l,\alpha}$ can be replaced by any vector bundle (not necessarily simple) $E\in \mathbb S_{A, y\nu\l}$ such that (\ref{basic1}) holds (for $\lambda=y\nu$).
\end{remark} 

\begin{remark} Let $\nu,\mu> \beta^i_A(\l)$ be rational numbers. By means of Corollary \ref{cor:fract} we can express the thresholds $s_A^i(\nu\l)$ (for $i=0,1$)  in function of the threshold $s^i_A(\mu\l)$:
\begin{equation}\label{eq:expression}
\frac {\nu s^i_A(\nu\l)}{1+s^i_A(\nu\l)}=\frac {\mu s^i_A(\mu\l)}{1+s^i_A(\mu\l)}
\end{equation}
This seems to be non-trivial already when $\nu$ and $\mu$ are integers. 
\end{remark}

Finally, we note that combining  Corollary \ref{exchange2}  and  Corollary \ref{cor:fract} one gets formulas relating $i$-th  thresholds for $A$ and  $(1-i)$-th thresholds for $\widehat A$ (for $i=0,1$). 
In particular, plugging Corollary \ref{cor:fract} (for $i=0$, on $(\widehat A,\hat\l)$) in Corollary \ref{exchange2}  (for i=1) we get the following result  (Corollary \ref{corollary} of the Introduction). It will be applied the next section.

\begin{corollary}\label{cor:corollary} For $\nu>\beta^0_{\widehat A}(\hat\l)$
\[
\beta^1_A(\l)=\frac{s^0_{\widehat A}(\nu\hat\l)+1} {s^0_{\widehat A}(\nu\hat\l)d_1d_g\nu} \>.
\]
\end{corollary}

\section{Bounding the thresholds} \label{bounds}

 The aim of this Section is to obtain a lower bound on the base point freeness threshold $\beta^1_A(\l)$. In view of the last Corollary, this is equivalent to an upper bound on the threshold $s^0_{\widehat A}(\nu\hat\l)$. In the next Lemma we show such an upper bound. 
 As usual the rank of a simple bundle $E_{A,\nu\l}\in\mathbb S_{A,\nu\l}$ is denoted $r_{A,\l}(\nu)$.

\begin{lemma}\label{relations}  If $\nu > \beta^0_{A}(\l)$ then 
\[
s^0_{A}(\nu\l)\le  {r_{A,\l}(\nu)}.
\]
\end{lemma}  
\begin{proof} To prove the assertion it is enough to show that, under the assumption $\nu>\beta^0_A(\l)$,  there is a point  $\alpha\in\widehat A$ such that the multiplication map of global sections
\[
m^{r_{A,\l}(\nu)}_{A,\nu\l,\alpha}:H^0(A,E_{A,\nu\l})\otimes H^0(A, E_{A, r_{A,\l}(\nu)\nu\l}\otimes P_\alpha)\rightarrow H^0(A,E_{A,\nu\l}\otimes E_{A, r_{A,\l}(\nu)\nu\l}\otimes P_\alpha)
\]
 is not injective (in fact this means that $EV^\bullet_{A,\nu\l}\otimes E_{A, r_{A,\l}(\nu)\nu\l}$ is not IT(1), see Remark \ref{rem:forgotten-remark}).

 Note that,  by definition of $\mathbb S_{A,\nu\l}$, the rational number $r_{A,\l}(\nu)\nu$ is  
 integer because $\frac{[\det E_{A,\nu\l}]}{r_{A,\l}(\nu)}=\nu\l$. Hence $E_{A, r_{A,\l}(\nu)\nu\l}$ is a line bundle representing 
$r_{A,\l}(\nu)\nu\l$, i.e. a line bundle of the form $\wedge^{r_{A,\l}(\nu)}E_{A,\nu\l}\otimes P_\beta$ for $\beta\in\widehat A$. Hence it is enough to prove that the map
\begin{equation}\label{eq:wedge}
 H^0(A,E_{A,\nu\l})\otimes H^0(A,\wedge^{r_{A,\l}(\nu)} E_{A,\nu\l})\rightarrow H^0(A,E_{A,\nu\l}\otimes \wedge^{r_{A,\l}(\nu)} E_{A,\nu\l})
\end{equation}
is not injective. 
 
  To this purpose we claim that this is true under the stronger assumption
\begin{equation}\label{eq:assumption}
\nu\ge \beta^1_A(\l)
\end{equation}
 (by Subsection \ref{subs:interpretation} this means that $E_{A,\nu\l}$ is generically globally generated). 

To prove what claimed, we first note that  the image of the composition
\begin{equation}\label{eq:compo}
\xymatrix{
\wedge^{1+r_{A,\l}(\nu)}H^0(A,E_{A,\nu\l})\ar[r]\ar[rd]&  H^0(A,E_{A,\nu\l})\otimes\wedge^{r_{A,\l}(\nu)} H^0(A,E_{A,\nu\l})\ar[d]^{1\otimes f}\\
 &H^0(A, E_{A,\nu\l})\otimes H^0(A, \wedge^{r_{A,\l}(\nu)} E_{A,\nu\l})}
\end{equation}
sits in the kernel  of the multiplication map (\ref{eq:wedge}). Indeed (\ref{eq:wedge}), restricted to the image of (\ref{eq:compo}), factors through $H^0(A,\wedge^{1+r_{A,\l}(\nu)}E_{A,\nu\l})$ which is zero.  

On the other hand we assert that the image of (\ref{eq:compo}) is non-zero. To prove this  we first note that  $\wedge^{1+r_{A,\l}(\nu)}H^0(A,E_{A,\nu\l})$ itself is non-zero. Indeed,  by assumption (\ref{eq:assumption}), the bundle $E_{A,\nu\l}$ is at least generically globally generated, hence $h^0(A,E_{A,\nu\l})=\chi(E_{A,\nu\l})\ge r_{A,\l}(\nu)$. If equality holds then Remark \ref{rem:equality} tells us that $\nu=\beta^0_{\widehat A}(\l)$, contrary to the hypothesis.  
Moreover, again because  $E_{A,\nu\l}$ is assumed to be generically globally generated, 
the map $f$ appearing in (\ref{eq:compo}) is non-zero. Dualizing (\ref{eq:compo}) we get
\begin{equation}\label{eq:compo-dual}
\xymatrix{H^0(A, E_{A,\nu\l})^\vee\otimes H^0(A, \wedge^{r_{A,\l}(\nu)} E_{A,\nu\l})^\vee\ar[r]^{1\otimes f^\vee}\ar[rd]&H^0(A,E_{A,\nu\l})^\vee\otimes\wedge^{r_{A,\l}(\nu)} H^0(A,E_{A,\nu\l})^\vee\ar[d]\\
&\wedge^{1+r_{A,\l}(\nu)}H^0(A,E_{A,\nu\l})^\vee},
\end{equation}
where the vertical arrow is the usual multiplication in the Grassmann algebra, and it is clear that such composition cannot be zero. This proves the assertion, and the claim.

As already mentioned, the claim yields that $s^0_{A}(\nu\l)\le  {r_{A,\l}(\nu)}$, i.e. the Proposition, under the additional assumption (\ref{eq:assumption}). 
If  such assumption is not satisfied then the rank of the evaluation map of global sections of $E_{A,\nu\l}$, say $s$, is strictly smaller than the rank of  $r_{A,\l}(\nu)$.  But then, by the same argument as above, the multiplication map of global sections
\begin{equation}\label{eq:wedge2}
 H^0(A,E_{A,\nu\l})\otimes H^0(A,\wedge^{s} E_{A,\nu\l})\rightarrow H^0(A,E_{A,\nu\l}\otimes \wedge^{s} E_{A,\nu\l})
\end{equation}
is not injective. The vector bundle $\wedge^{s} E_{A,\nu\l}$ is semihomogenous, and, being a direct summand of $E_{A,\lambda\l}^{\otimes s}$, belongs to $\mathbb S_{A,s\nu\l}$.  Moreover it satisfies (\ref{basic1})  (this follows from the same Proposition \ref{basic}). Therefore  $s^0_A(\l)\le s$ because, by Remark\ \ref{rem:any}, we can replace $E_{A,s\nu\l}$ with  $\wedge^{s} E_{A,\nu\l}$ in the definition of $s^0_A(\l)$.
 Since $s< r_{A,\l}(\nu)$ the Proposition is proved also in this case. 
\end{proof}
 
As a consequence we have the following upper bound for the base point freeness threshold, slightly more precise than  Theorem \ref{obstruction} of the Introduction. 

\begin{corollary}\label{cor:precise}
\[
\beta^1_A(\l)\ge \sup_{\nu>\beta^0_{\widehat A}(\hat\l)}\Bigl\{\frac{1+r_{\widehat A,\hat\l}(\nu)} {d_1d_g\nu\, r_{\widehat A,\hat\l}(\nu)}\Bigr\}.
\]
\end{corollary}
\begin{proof}
We have
\[\frac 1 {d_1d_g\beta^1_A(\l)}=\beta^0_{\widehat A}(\hat\l)=\frac{\nu s^0_{\widehat A}(\nu\hat\l)}{1+s^0_{\widehat A}(\nu\hat\l)}\le \frac{\nu r_{\widehat A,\hat\l}(\nu)}{1+r_{\widehat A,\l}(\nu)}
\]
where the first equality is Corollary \ref{exchange2}, the second one follows from Corollary \ref{cor:fract} and the last inequality follows from Lemma \ref{relations} applied to the dual polarized variety $(\widehat A,\hat\l)$ (under the hyothesis $\nu>\beta^0_{\widehat A}(\hat\l)$)).
\end{proof}

Theorem \ref{obstruction} in the Introduction follows because in any case   $\frac 1 {\sqrt[g]{\chi(\hat\l)}}\ge \beta_{\widehat A}^0(\hat\l)$. Indeed, let $\nu\in\mathbb Q^+$ such that $\nu\le\beta^0_{\widehat A}(\hat\l)$. Then, from Proposition \ref{interpret}(4) it follows that 
 $h^0(\widehat E_{\widehat A,\nu\hat\l})=\chi(E_{\widehat A,\nu\hat\l})\le r_{\widehat A,\hat\l}(\nu)$. Then  it follows that $\frac 1{\chi(\hat\l)}\ge \nu^g$ (Subsection \ref{subs:semihom}(b)). 

Now  we turn to the proof of Corollary \ref{cor}. Let $(d_1,\dots ,d_g)$ be the type of $\l$. Hence the type of $\hat\l$ is $(d_1,\frac {d_1d_g} {d_{g-1}},\dots ,\frac{d_1d_g}{d_2},d_g)$. 
The inequality of Corollary \ref{cor}(1) follows from Corollary \ref{cor:precise} (or Theorem \ref{obstruction}) for $\nu=\frac 1 {d_1}$ (which contributes to the $\sup$ as soon as $d_g>d_1$). In fact in any case $r_{\widehat A,\hat \l}(\frac 1 {d_1})=1$ (and $E_{\widehat A,\frac 1 {d_1}\hat\l}$ is a line bundle  representing the  polarization $\frac 1 {d_1}\hat\l$). \\
The inequality of Corollary \ref{cor}(2) follows taking $\nu=\frac 1 m$. In the case at hand  the type of $\hat\l$ is $(1,m,\dots ,m,mn)$ and therefore, according to (\ref{rk/dim}), $r_{\widehat A,\hat\l}(\frac 1 m)=m$.\\

\section{Problems, remarks and variants}

\subsection{Approach to upper bounds}\label{subs:GG}  Although in this paper we used semihomogeneous vector bundles to produce a lower bound for the base point freeness threshold $\beta^1_A(\l)$, it seems plausible that they can provide an approach to the search of upper bounds for $\beta^1_A(\l)$ too. Such upper bounds lead to criteria ensuring projective normality and the $N_p$ conditions. In fact, via Proposition \ref{interpret}, this would follow from  (generic) global generation criteria for ample  semihomogeneous vector bundles. For example, a first natural case to consider is when the rank $r_{A,\l}(\nu)=2$ and $\chi(E_{A,\nu\l})\ge 2$.   Assuming that the polarization is primitive, i.e. that $d_1=1$, by (\ref{rk/dim}) this happens if and only if $\nu=\frac 1 2$ and the type of the polarization has the form $(1,d_2,\dots d_g)$ with $d_2$ even and $d_g>2$ . 

\subsection{Jump/degeneracy/base loci}\label{base-loci} An important role in the theory of cohomological rank functions is played by the dimension of cohomological jump loci of $\mathbb Q$-twisted sheaves (see \cite[Proposition 4.4]{ens}). As mentioned in Subsection \ref{IT}, the dimension of such loci is well defined, but the actual loci are defined only up to pulling back or pushing forward via isogenies and/or translations.  However there are instances where  realizations of jump loci, not only their dimension, would be important. Such realizations can be produced (up to translation) via simple semihomogeneous vector bundles. For example if the base point freeness threshold $\beta^1_A(\l)$ is a rational number, say $\lambda$, by Proposition \ref{fractional} one can take take as realization of (the support of) the base locus the pullback via the isogeny $\varphi_\l$ of the degeneracy locus of the evaluation map $H^0(A,E_{A,\lambda\l})\otimes\OO_A\rightarrow E_{A,\lambda\l}$. 

Another example of the interest of jump loci in our setting is in relation with projective normality. The key condition for projective normality is the surjectivity of the multiplication map of global sections $H^0(A,L)\otimes H^0(A,L)\rightarrow H^0(A,L^{\otimes 2} )$. As we already mentioned, this is equivalent to the fact that the identity point $\hat e\in\widehat A$ does not belong the cohomological support locus $V^1(A, EV^\bullet_{A,L}\otimes L)$. At present, it is not clear when this happens. The question is not vacous only when $V^1(A, EV^\bullet_{A,L}\otimes L)$ is a proper subvariety, and we know that this happens  if and only if $\beta^1_A(\l)=\frac 1 2$. In turn, via Corollary \ref{cor:corollary}, this is equivalent to $s^0_{\widehat A}(\nu\hat\l)=\frac 2 {d_1d_g\nu-2}$ for some $\nu$. This in turn means that the cohomological support locus $V^0(\widehat A,EV^\bullet_{\widehat A,\nu\l}\otimes E_{\widehat A,\frac 2 {d_1d_g\nu-2}\nu\hat\l})$ is non-empty. The question is whether one can follow the various steps of the proof of Corollary \ref{cor:corollary} in order to describe the former cohomological support locus in function of the latter. In this way one could understand whether the polarization $\l$ is or is not projectively normal. 

\subsection{Jets-separation and higher gaussian maps. } The results of Section \ref{sect:2} are completely general, and they can be applied to other interesting situations.  For example, as pointed out by Alvarado in \cite{nelson}, for each ideal sheaf $\I$ in a polarized abelian variety it is natural to consider the \emph{vanishing threshold} \[\inf\{\lambda\>|\> \I\langle\lambda\l\rangle\>\hbox{is IT0}\}.
\]
When $\I=\I_e$ is the ideal of the origin one recovers the base point freeness threshold $\beta^1_A(\l)$. Other approachable examples are for $\I=\I_e^{k+1}$, where one gets the \emph{k-jets separation thresholds}, here denoted $\beta^1_{A,k}(\l)$. Via Proposition \ref{fractional} this is translated as follows
\[
\beta^1_{A,k}(\l)=\inf\{\lambda\in\mathbb Q^+\>|\>\I_e^{k+1}\otimes E_{A,\lambda\l}\>\> \hbox{is IT0}\}.
\]
As in Subsection \ref{subs:interp}, an upper bound for such thresholds, say $\beta^1_{A,k}(\l)\le \lambda$, can be translated in  the generic $k$-immersivity of the semihomogenous bundle $E_{A,\lambda\l}$ for $\lambda\ge\beta^1_{A,k}(\l)$, and equality holds if and only if $E_{A,\lambda\l}$ is generically $k$-immersive but not immersive. 

Such thresholds are very precisely related to the Seshadri constants, here denoted $\varepsilon_A(\l)$,  as follows (\cite[Theorem 4.1]{nelson}):
\begin{equation}\label{k-jets}
\varepsilon_A(\l)=\sup_k\frac{k+1}{\beta^1_{A,k}(\l)}=\lim_{k\rightarrow\infty}\frac{k+1}{\beta^1_{A,k}(\l)}.
\end{equation}

In parallel with the correspondence between the base point freeness threshold  and the threshold $s^1_A(\l)$ for surjectivity of multiplication maps of global sections, the FM transform induces  a correspondence between $k$-jets separation thresholds and thresholds for the surjectivity of higher gaussian maps, here denoted $s^1_{A,k}(\l)$  (see \cite[\S6]{nelson}). Again, they can be computed by the failure of surjectivity of higher gaussian maps of specific semihomogeneous vector bundles. This is  useful in producing explicit equations for the loci of polarized abelian varieties whose Seshadri constant satisfies a given upper bound. An example of this is  in \cite[Subsection 5.3]{ap}.


\providecommand{\bysame}{\leavevmode\hbox
to3em{\hrulefill}\thinspace}


\begin{thebibliography}{EMS}

 \bibitem[A]{nelson} N. Alvarado, {Jets-separation thresholds, Seshadri constants and higher Gauss-Wahl maps on abelian varieties},  IMRN  no. 9, Paper No. rnaf107 (2025) 24 pp.

\bibitem[AP]{ap} N. Alvarado, G. Pareschi, {Abelian varieties analogs of two results about algebraic curves}, Math. Ann. 395, 74 (2026) 

\bibitem[BL]{birke-lange} Ch. Birkenhake and H. Lange, {\em Complex abelian varieties}, 2nd edition, Springer, 2004.

\bibitem[C1]{caucci} F. Caucci, {The basepoint-freeness threshold and syzygies of abelian varieties}, Algebra  Number Theory 14 (2020) 947--970

\bibitem[C2]{caucci3} F. Caucci, Syzygies of Kummer varieties, Trans. Amer. Math. Soc. 377 (2024),  1357 -- 1370

\bibitem[C3]{caucci2} F. Caucci, {Higher order embeddings via the basepoint-freeness threshold},  Proc. Amer. Math. Soc.
152 (2024) 3623--3628

\bibitem[DJZ]{zhi3} J. Du, Z. Jiang, G. Zhang, {Cohomological rank functions and surfaces of general type with $p_g=q=2$}, Math. Ann. 393 (2025) 1--38

\bibitem[F]{fuentes} L. Fuentes Garcia, {Some results about the projective normality of abelian varieties}, Arch. Math. (Basel) 85 (2005) 409--418

\bibitem[H]{hacon} Ch. Hacon, {A derived category approach to generic vanishing}, J. Reine Angew. Math. 575 (2004), 173--187.


\bibitem[I1]{ito-0} A. Ito, {A remark on higher syzygies on abelian surfaces},
Comm. Algebra 46 (2018)  5342-5347

\bibitem[I2]{ito-M-reg} A. Ito, {M-regularity of $\mathbb Q$-twisted sheaves and its application to linear systems on abelian varieties}, Trans. Amer. Math. Soc. 9 (2022) 6653--6673

\bibitem[I3]{ito-3-folds} A. Ito, {Basepoint-freeness thresholds and higher syzigies of abelian threefolds}, Algebraic Geometry 9 (2022) 762--787

\bibitem[I4]{ito3} A. Ito, {Higher syzygies of general polarized Abelian varieties of type $(1,\dots ,1,d)$}, Math. Nachr.  296 (2023) 3395--3410

\bibitem[I5]{ito4} A. Ito, {Projective normality and basepoint-freeness thresholds of general polarized abelian varieties}, Bull.  Lond. Math. Soc. 55 (2023) 2793--2816


\bibitem[J1]{zhi-survey} Z. Jiang, {Cohomological rank functions and syzygies of abelian varieties}, Math. Z. 300 (2022) 3341--3355

\bibitem[J2]{zhi2} Z. Jiang, {Anghern-Siu-Helmke's method applied to abelian varieties}, Forum Math. Sigma 11 (2023) Paper No. e38, 16 pp.



\bibitem[JP]{ens} Z. Jiang,  G. Pareschi, {Cohomological rank functions on abelian varieties}, Annales de L'Ecole Normale Superieure 53 (2020), 815 - 846



\bibitem[L] {laz2} R. Lazarsfeld, {\em Positivity in algebraic geometry. II. Positivity for vector
bundles, and multiplier ideals. } Ergebnisse der Mathematik und ihrer Grenzgebiete. 3. Folge [Results in Mathematics and Related Areas. 3rd Series], 49. Springer-Verlag, Berlin, 2004.


\bibitem[LR]{lr} M. Lahoz, A. Rojas, {Chern degree functions}, Commun. Contemp. Math. 25 (2023) Paper No. 2250007, 53 pp.

 \bibitem[Lo]{lo} V. Lozovanu,  {Singular divisors and syzygies of abelian threefolds},  arXiv:1803.08780 (2018)

\bibitem[M1]{semihom} S. Mukai, {Semi-homogeneous vector bundles on abelian varieties} J. Math. Kyoto Univ. 18 (1978), 239--272

\bibitem[M2]{mukai} S. Mukai,  {Duality between D($X$) and D($\hat X$) with its application to Picard sheaves}, Nagoya Math. J. \textbf{81} (1981), 153--175.

\bibitem[Mu]{mumford} D. Mumford, {\em Abelian varieties}, Second Edition, Oxford University Press, 1970.


\bibitem[NR]{nr} D.S. Nagaraj, S. Ramanan, {Polarizations of type $(1,2,\dots ,2)$ on abelian varieties}, Duke Math. J. 80 (1995), 157--194

\bibitem[O]{oprea} D. Oprea, {The Verlinde bundles and the semihomogeneous Wirtinger duality}, J. Reine Angew. Math. 654 (2011), 181--217

\bibitem[P]{n-torsion} G. Pareschi, {Torsion points on theta divisors and semihomogeneous vector bundles}, Algebra Number Theory 15 (2021) 1581--1592










\bibitem[PP]{pp2} G. Pareschi, M. Popa, {GV-sheaves, Fourier-Mukai transform, and Generic Vanishing}, Amer. J. of Math. \textbf{133} (2011), 235--271.



\bibitem[R]{rojas}  A. Rojas, {The basepoint-freeness threshold of a very general abelian surface}, Select Math. (N.S.) 28 (2022) Paper No. 34, 14 pp.

\bibitem[Ru]{rubei} E. Rubei, {Projective normality of abelian varieties with a line bundle of type (2, . . . )}. Boll. Un. Mat. Ital.
Sez. B Artic. Ric. Mat. 8 (1998) 361–367 












\end{thebibliography}
\end{document}